\documentclass[11pt]{amsart}
\usepackage{amssymb}

\theoremstyle{plain}
\newtheorem{theorem}{Theorem}[section]
\newtheorem{lemma}[theorem]{Lemma}
\newtheorem{proposition}[theorem]{Proposition} 
\newtheorem{corollary}[theorem]{Corollary}
\newtheorem{observation}[theorem]{Observation}

\theoremstyle{remark}
\newtheorem{remark}[theorem]{Remark}

\theoremstyle{definition}
\newtheorem{definition}[theorem]{Definition} 
\newtheorem{notation}[theorem]{Notation} 
\newtheorem*{assumption}{Assumption} 

\newtheorem{example}[theorem]{Example}

\newcommand{\C}{\mathbf{C}}
\newcommand{\Q}{\mathbf{Q}}
\newcommand{\Z}{\mathbf{Z}}
\newcommand{\Ps}{\mathbf{P}}

\newcommand{\Aoldone}{A2}
\newcommand{\Aoldtwo}{A1}
\newcommand{\Adwa}{A1}
\newcommand{\Ajeden}{A0}
\newcommand{\dedelta}{{\mathfrak{d}}}
\newcommand{\FX}{F}
\newcommand{\fgg}{f_{0}}
\newcommand{\hhh}{f_{1}}
\newcommand{\segcub}{C_{3}}
\newcommand{\Cppp}{C'_{3}}
\newcommand{\Cbbb}{C''_{3}}
\newcommand{\aaa}{a_1}
\newcommand{\bb}{a_2}
\newcommand{\cc}{a_3}
\newcommand{\dd}{a_4}
\newcommand{\ee}{b_3}
\newcommand{\ff}{b_4}

\DeclareMathOperator{\CH}{CH}

\DeclareMathOperator{\sing}{sing}
\DeclareMathOperator{\red}{red}

\DeclareMathOperator{\rank}{rank}

\DeclareMathOperator{\Pic}{Pic}

\DeclareMathOperator{\sat}{sat}

\DeclareMathOperator{\Jac}{Jac}

\numberwithin{equation}{section}

\title{Quintic threefolds with triple points}

\author[R.~Kloosterman]{Remke Kloosterman}
\address{Universit\`a degli Studi di Padova,
Dipartimento di Matematica,
Via Trieste 63,
35121 Padova, Italy} 
\email{klooster@math.unipd.it}
\author[S.~Rams]{Slawomir Rams}
\address{Institute of Mathematics, Jagiellonian University,
 ul.~{\L}ojasiewicza~6, 30-348 Krak\'ow, Poland} 
\email{slawomir.rams@uj.edu.pl}
\date{\today}
\thanks{S.~R. is partially supported  by National Science Centre, Poland, grant 2014/15/B/ST1/02197. 
Part of this research was completed during a visit financed by the programme ``Mobilit\`a docenti" of the University of Padova}
\begin{document}

\begin{abstract}
We study the geometry of quintic threefolds  $X\subset \Ps^4$  with only ordinary triple points as singularities. In particular, we show that if a quintic threefold $X$  has a reducible hyperplane section then $X$ has at most $10$ ordinary triple points, and that this bound is sharp.

We construct various examples of quintic threefolds with triple points and discuss their defect. 
\end{abstract}

\maketitle

\section{Introduction}\label{secIntro}

The main aim of this note is to study three-dimensional quintic hypersurfaces with ordinary triple points as singularities. The interest in quintic threefolds with isolated singularities can be partially justified by the fact that their canonical class is trivial. That is why there is extensive literature on three-dimensional nodal quintics (see Meyer's book \cite{MeyBook}) and their small resolutions. Hardly anything is known on quintics with higher singularities.

Let $X\subset \Ps^4$ be a quintic hypersurface with only ordinary triple points as singularities. Then the blow-up of $X$ at its triple points is a Calabi-Yau threefold. 
One can show (see Corollary~\ref{cor-def10}) that if $X$ has at least 10 triple points, then $X$  is not $\Q$-factorial, i.e., $X$ contains several  Weil divisors which are not $\Q$-Cartier. 
We investigate which implications the existence of such Weil divisors has on the geometry of $X$. In particular, we show the following theorem:

\begin{theorem} \label{thm-no-eleven}
 Let $X$ be a quintic threefold with only ordinary triple points as singularities. If a hyperplane section of  $X$ is reducible, then $X$ has at most 10 singular points.
\end{theorem}
Moreover, we exhibit several examples of quintic threefolds with 10 ordinary triple points, each of which contains a $2$-plane. Thus Theorem~\ref{thm-no-eleven} gives a  sharp bound
(for the discussion of the same bound under weaker assumptions - see Remark~\ref{remark-strongerclaim}).

We are not aware of any example of a quintic threefold with  only ordinary triple points as singularities which is not $\Q$-factorial and contains no $2$-planes.

To prove Theorem~\ref{thm-no-eleven} we show  that  the quintic $X$ contains a $2$-plane $\Pi_0$ (Proposition~\ref{prpBasic}.\ref{p31-eight}). Then we study the geometry of the projection from 
 $\Pi_0$.
We show that if $X$ has 11 singularities  then the projection has at least three reducible fibers such that  their Zariski-closures contain additional $2$-planes (Proposition~\ref{corPlanes}). Then
we study  the projections from those extra $2$-planes to show that 10 of the 11 triple points form the so-called Segre configuration (Proposition~\ref{prop-segreconf}). 
Finally, we prove that a quintic threefold with ordinary triple points at the Segre configuration cannot have an 11th ordinary triple point (Lemma~\ref{lem-no11}). 

In general it is unknown what is the maximal number of isolated singularities of a given type on a quintic threefold. For $A_1$ singularities the best examples were given by Hirzebruch and  Van Straten (see \cite{DucoQuintic})
and the best upper bound results from  Varchenko's spectral bound \cite{VarSpec}. 
Moreover, \cite{VarSpec} implies that there are no quintic threefolds  with more than 11 ordinary triple points. As a by-product our investigation shows that a hypothetical quintic threefold $X$  with 11 triple points contains no low-degree surfaces ($2$-planes, quadrics and cubics) but its defect is at least 20 (see Corollary~\ref{cor-def10}), which means that it contains many Weil divisors 
 (i.e.,  $\rank \CH^{1}(X)/\Pic(X) \geq  20$). 

The organisation of this paper is as follows:
In Section~\ref{secInv} we determine the main numerical invariants of a quintic hypersurface $X \subset \Ps^4_{\C}$ with ordinary triple points. Section~\ref{secBas} contains various  restrictions on the position of triple points on the quintic $X$. In Section~\ref{secPlane} we study the geometry of the quintic threefold $X$ (resp. its hyperplane sections) under the additional assumption that $X$ contains a $2$-plane $\Pi_0$.  
Section~\ref{secSegre} is devoted to  a hypothetical quintic $X$  with 11 triple points that contains a $2$-plane. We use the results from the previous section to conclude that there are at least five further $2$-planes on $X$.
The projections from the five planes enable us to prove that 10 of the 11 triple points form the so-called Segre configuration.
In Section~\ref{sec11} we show that a quintic threefold with 10 ordinary triple points, which form the Segre configuration cannot have an 11th ordinary triple point.
Section~\ref{secExa} contains various examples: we discuss the possible values for the defect $\delta(X)$ for a quintic threefold $X$ with at most 6 triple points, we give various examples with 7, 8 and 9 triple points and construct three families of quintic threefolds with 10 triple points and distinct values for the defect. \\
The considerations of Sections~\ref{secPlane} and~\ref{secSegre} are partially based on the (elementary and tedious)  study of possible degenerate fibers 
of the resolution of the projection from the plane $\Pi_0$.  In Section~\ref{secFib} we examine the degenerate fibers. 
Finally in Appendix~\ref{appPen} we prove some elementary results on pencils of quartic curves that we need in Section~\ref{secFib}. \newline

\section{Invariants of quintic threefolds}\label{secInv}

Let  $X \subset \Ps^4_{\C}$ be a quintic hypersurface. 
In this note we work under the following assumption

\begin{assumption} {\bf [\Ajeden]}  
The complex quintic threefold   $X$ is assumed to have  \emph{only finitely many singularities} $P_1$, $\ldots$, $P_t$ all of which
are \emph{ordinary triple points}. 
\end{assumption}

Recall that a triple point $P \in X$ is called   an ordinary triple point if and only if the  tangent cone of $X$ at $P$ is a cone over a smooth cubic surface.

The Milnor number of an ordinary triple point is 16, whereas its Tjurina number is either 15 or 16. 
A smooth quintic threefold has Euler number $(-200)$.  Therefore, by \cite[Corollary 5.4.4]{Dim}, we have 
\begin{equation}   \label{eq-euler-char-singular}
\chi(X) = -200+16t.
\end{equation}

\begin{definition} Let $Y\subset \Ps^4$ be a hypersurface with isolated singularities. We call the integer \[\delta := \delta(Y) := h^4(Y)-1\]
the \emph{defect} of $Y$.
\end{definition}

The Lefschetz theorem \cite[Theorem 5.2.6]{Dim} yields that $h^i(X,\C) = h^i(\Ps^4,\C)$ for $i < 3$. The same equality holds for $i=5,6$ by  \cite[Theorem 5.2.11]{Dim}.
Thus we have that $H^i(X,\C) = 0$ for $i \neq  0,2,3,4,6$ and $b_i(X)=1$ for $i=0,2,6$. By definition of defect we have $b_4(X)=1+\delta$.
Finally,  \eqref{eq-euler-char-singular} yields $b_3(X)= 204 + \delta - 16t$.
 Moreover, the isomorphisms in the Lefschetz theorem preserve 
the mixed Hodge structure on $H^i(X,\C)$, so we have 
$H^{2k}(X,\C)=\C(-k)$ for  $k=0,1$ (i.e. $H^{2}(X,\C)=H^{1,1}(X)= \C$). 

We put $\tilde{X}$ (resp. $E$)  to denote the blow-up of  $X$ along $\sing(X)$ (resp. its exceptional divisor) 
and apply the standard Leray spectral sequence to see that $h^i({\mathcal O}_{\tilde X}) = 0$ for $i =1, 2$.
One can easily check that  $\tilde{X}$ is smooth and each point $P_j \in \sing(X)$ is replaced by a smooth cubic surface on $\tilde{X}$. 
Hence $\chi(\tilde{X})=-200+24t$.

Thanks to the vanishing of certain Betti numbers of $X$ and  $E$, the long exact sequence  
\cite[Corollary-Definition 5.37]{PetersSte} yields the exact sequence:
\begin{equation} \label{eq-lsshort}
0 \longrightarrow H^4(X) \longrightarrow H^4(\tilde{X}) \longrightarrow H^4(E) \longrightarrow 0 
\end{equation}
of mixed Hodge structures. 
By standard arguments we have  $H^4(\tilde{X})= H^{2,2}(\tilde{X}) \simeq H^{1,1}(\tilde{X})$, the other Hodge numbers $h^{p,q}(\tilde{X})$ vanish for $p+q=4$, whereas $H^4(E)= H^{2,2}(E) \simeq  \C^t$.
Thus, by \eqref{eq-lsshort}  the only non-zero Hodge number of $H^4(X)$ is $h^{2,2}$ and we have the equality:
\begin{equation} \label{eq-deltaonetwo}
1+\delta = h^{2,2}(H^4(X)) = h^{2,2}(\tilde{X}) - h^{2,2}(E) =  h^{2,2}(\tilde{X}) - t
\end{equation}

Our considerations yield the following proposition.
\begin{proposition} \label{prop-hodgetildeX} The following equalities hold
$$
h^{3,0}(\tilde{X})=1, \quad h^{2,1}(\tilde{X})=101-11t+\delta \mbox{ and } h^{1,1}(\tilde{X})=1+t+\delta.
$$ 
\end{proposition}

To compute the Hodge numbers of $H^3(X)$ observe that from \cite[Corollary-Definition 5.37]{PetersSte} one obtains the following exact sequence of mixed Hodge structures
\begin{equation*}
0 \longrightarrow H^2(X) \longrightarrow  H^2(\tilde{X})  \longrightarrow H^2(E)   
\longrightarrow H^3(X) \longrightarrow  H^3(\tilde{X})  \longrightarrow 0    \, .
\end{equation*}
Moreover, the Hodge structure on  $H^j(E)$ can be easily determined, so Proposition~\ref{prop-hodgetildeX} and \eqref{eq-deltaonetwo}
imply the following proposition.
\begin{proposition} \label{prp-formulae}
For $k=0,1,3$ we have that $H^{2k}(X,\C)=\C(-k)$, whereas  $H^4(X,\C)\cong \C(-2)^{1+\delta}$.
Moreover, the following equalities hold
\[  h^{0,3}(H^3(X))=h^{3,0}(H^3(X))=1; h^{1,2}(H^3(X))=h^{2,1}(H^3(X))=101-11t+\delta\]
and $h^{1,1}(H^3(X))=6t-\delta$.
All other Hodge numbers vanish.
\end{proposition}
As a consequence of the non-negativity of  $h^{2,1}(H^3(X))$ and Proposition~\ref{prp-formulae}  we obtain the following corollary (c.f. Example~\ref{example-sevthree}). 
\begin{corollary} \label{cor-def10} If $X$ satisfies {\rm  [\Ajeden]}, then  $\delta(X) \geq 11t - 101$.
\end{corollary}

Hence to determine the Hodge numbers of both $X$ and $\tilde{X}$ it suffices to determine $\delta$. 
Observe, that  \eqref{eq-deltaonetwo} combined with \cite[Corollary~6]{CynkTriple} implies  that the defect $\delta$ coincides with the integer that is introduced in   \cite[Definition~2]{CynkTriple}. 
\begin{proposition}\label{prpDefectFormula}
Let $G\in \C[x_0,\dots,x_4]_d$ be such that $Y:=V(G)$ is a hypersurface with only ordinary triple points at $P_1,\dots,P_t$. If $J$ is the ideal 
\begin{equation} \label{eq-def-ideal}
J :=  ((\cap_{i=1}^t I(P_i)^ 3)+\Jac(G))^{\sat}.
\end{equation}
then the following equality holds 
\[ \delta= 11t - h_J(2d-5), \] 
where $h_J(\dot)$ is the Hilbert function of $J$.
\end{proposition}

\begin{proof}
By \cite[Corollary~6]{CynkTriple}
the integer $\delta$ is the defect in degree $(2d-5)$ of the ideal $J'$ that can be constructed as follows: \\
Let $P$ be an ordinary triple point of $Y$ and  let $f_P(z_1,z_2,z_3,z_4)=0$ be a local equation of $Y$ around $P$. 
We put ${\mathfrak m}_P$ to the denote the ideal of the point $P$ and
define $\mbox{I}_P$ as the ideal  generated by ${\mathfrak m}_P^3$ and all partials of $f_P$. Finally, we identify $\mbox{I}_P$ with its pullback to $\C[x_0,\dots,x_4]$ and define  $J':=\cap_{P  \in \sing(Y)} \mbox{I}_P$. By construction, $J'$ is a saturated ideal of length  $11t$.
 
We claim that $J' = J$. Indeed, observe that the localization of $J$ at $P$ is  generated by ${\mathfrak m}_P^3$ and the $(n+1)$ partial derivatives of $G$.
 Among these derivatives there are $n$ which locally generate the same ideal as the partials of $f_P$. Using the Euler relation for $G$ we see that the ideal generated by  its $n+1$ partials coincides with  
 the ideal generated by $f_P$ and its $n$ partial derivatives. 
By assumption, we have  $f_P\in {\mathfrak m}_P^3$, so both $I_P$ and the localization of $J$ at $P$ coincide. Hence $J$ and $J'$ are saturated ideals that define the same scheme and therefore they are equal.
\end{proof}

Finally, we use Varchenko's spectral bound \cite{VarSpec} to bound the number of triple points of $X$.

\begin{lemma} The number of triple points of $X$ cannot exceed  11.\label{lemSpec}
\end{lemma}
\begin{proof}
The spectrum of an ordinary triple point is
\[ (\frac{1}{3})+4(\frac{2}{3})+6(1)+4(\frac{4}{3})+(\frac{5}{3})  , \]
whereas the spectrum of an ordinary fivefold point equals
\[ (\frac{-1}{5})+4(0)+10(\frac{1}{5})+20(\frac{2}{5})+31(\frac{3}{5})+40( \frac{4}{5})+44(1)+40(\frac{6}{5})+\dots\]
By \cite[Thickeness Thm]{VarSpec}, 
for any $\alpha$ the interval $(\alpha,\alpha+1)$ is a semincontinuity set for the spectrum (for any arbitrary lower deformation). For $\alpha=2/5$ the spectrum of the fivefold point 
has length $31+40+44+40=155$, whereas the spectrum of the triple point has length $14$ 
in the interval $(2/5,7/5)$. Thus $X$  has at most $\frac{155}{14}=11+\frac{1}{14}$ triple points.
\end{proof}

\section{Basic properties}\label{secBas}
In this section we maintain the notation of Section~\ref{secInv}. In particular, $X$ satisfies [\Ajeden].
Moreover, we put $\FX$ to denote the \emph{generator of the ideal} $\mbox{I}(X) \subset \C[x_0, \ldots, x_4]$.  
The projective subspace of $\Ps^4(\C)$   spanned by $P_1, \ldots, P_k$  is denoted 
by $\langle P_1, \ldots, P_k \rangle$.

Below we collect some basic facts on the configuration of the  triple points of the quintic $X$.

\begin{proposition}\label{prpBasic}
Let $X$ be a quintic threefold with ordinary triple points $P_1$, $\ldots$, $P_t$ as its only singularities. 
\begin{enumerate}
 \item  \label{p31-one} If $j_1 \neq j_2$, then the line 
 $\langle P_{j_1}, P_{j_2} \rangle$  is contained in $X$.
 \item  \label{p31-two}  No three triple points of $X$ are collinear.
 \item \label{p31-three} A two-plane contains at most four triple points of $X$.
 \item \label{p31-four} A two-plane $\Pi$ contains four triple points of $X$ 
 if and only if 
$\Pi \subset X$.
\item \label{p31-six} If $H$ is a hyperplane and  $\sing(X)\cap H = \{P_1, \ldots, P_k\}$,  then  $k \leq 6$. Moreover, if $k=6$ then $X\cap H$ contains a $2$-plane.
\item \label{p31-five} Let $\Pi_1, \Pi_2, \Pi_3 \subset  X$ be distinct 2-planes. Then $\Pi_1\cap \Pi_2\cap \Pi_3$ is finite.

\item \label{p31-seven} If $X$ contains  an irreducible quadric surface $S_2$ then 
precisely 6 triple points of $X$ lie on $S_2$. In particular, $X$ contains a $2$-plane.
\item \label{p31-eight} If $H$ is a hyperplane then $X\cap H$ is reduced. If it is reducible, then it contains at least one  $2$-plane and at most three $2$-planes.
 \end{enumerate}

\end{proposition}
\begin{proof}

\noindent
(\ref{p31-one}): Let $\ell := \langle P_{j_1}, P_{j_2} \rangle$. 
The degree $5$ polynomial $\FX|_{\ell}$ has triple roots  at both $ P_{j_1}$ and  $P_{j_2}$, so it vanishes along $\ell$. 

\noindent (\ref{p31-two}):  We can assume that  $P_{j_1} = (1\colon0\colon\ldots\colon 0)$ and  $P_{j_2} = (0\colon 1\colon 0\colon \ldots\colon 0)$.
Then $\ell := \langle P_{j_1}, P_{j_2} \rangle = V(x_2, x_3, x_4)$ and,  by (1), we have
$\FX= \sum_{j=2}^{4} x_j f_j$. \\
Obviously the partials  $\partial \FX/\partial x_0$,  $\partial \FX/\partial x_1$  vanish along $\ell$. Moreover, for $j \geq 2$ 
the choice of the points $P_{j_1}$, $P_{j_2}$ implies that 
there exist $\alpha_j \in \C$ such that
$$
\partial \FX / \partial x_j =  f_j  = \alpha_j (x_0 x_1)^2 \mbox{ modulo the ideal } \mathrm{I}(\ell).
$$
Suppose that $X$ has  another triple point on $\ell$, say  $P_{j_3}=(c_0:c_1:0:0:0)$. 
Since $c_0 c_1 \neq 0$, all the coefficients $\alpha_2, \ldots, \alpha_4$ vanish. This yields 
$\ell \subset \sing(X)$ and contradicts the assumption that $\sing(X)$ is finite.

\noindent (\ref{p31-three}) and (\ref{p31-four}):
Suppose that $\Pi = V(x_0,x_1)$ and  $\sing(X) \cap \Pi = \{P_1, \ldots, P_k\}$.

 If $\Pi$ is not contained in $X$ then the intersection $X \cap \Pi$ is a quintic curve. 
By (\ref{p31-one}) it contains all the lines $\langle P_{j_1}, P_{j_2} \rangle$ for distinct  $j_1, j_2 \leq k$. From (\ref{p31-two}) we obtain   
 the inequality  $k \leq 3$.

Assume now that $\Pi$ is contained in $X$. Then we can write 
\begin{equation} \label{eq-expx0x1}
\FX=x_0 \fgg +x_1 \hhh .
\end{equation}
Observe that the hypersurface $V(\fgg)$ (resp. $V(\hhh)$) meets the plane $\Pi$ along a quartic curve and 
the points $P_1$, $\ldots$, $P_k$ are double points of both quartic curves. Thus for $j \leq k$ the two planar quartic curves meet at each  $P_j$ with multiplicity at least  four and the claim (\ref{p31-three}) follows from B\'ezout. 

To complete the proof of (\ref{p31-four}) it suffices to   show that  the intersection multiplicity of the curves given on $\Pi$  by $\fgg$ and $\hhh$ 
at $P_1 := (0\colon 0\colon 0\colon 0\colon 1)$   is precisely four.  We consider the expansion 
$f_{j}=\sum_{i=2}^4 f_{j,i}(x_0,x_1,x_2,x_3)x_4^{4-i}$
for $j=0,1$.
Then the tangent cone of $X$ at $P_1$ is given by 
\begin{equation} \label{eq-tang-cone}
g:=(x_0 f_{0,2}+x_1 f_{1,2})
\end{equation} 
Obviously the partials $\partial g / \partial x_2$, $\partial g / \partial x_3$ vanish along $V(x_0,x_1)$, whereas for $j=0,1$
we have 
\[ (\partial g / \partial x_j)|_{V(x_0,x_1)} = f_{j,2}(0,0,x_2,x_3).\]
By assumption $P_1$ is an ordinary triple point, so $g \in \C[x_0, \ldots, x_3]$ defines a smooth cubic surface.
Thus $f_{0,2}(0,0,x_2,x_3)$ and $f_{1,2}(0,0,x_2,x_3)$ have no common factor and therefore the intersection multiplicity of the planar quartic curves at $P_1$ is exactly four
(observe that the tangent cone of $V(f_{j}) \cap \Pi$ at $P_1$ is given by $f_{j,2}(0,0,x_2,x_3)$). Moreover, for every $(\lambda_0 \colon \lambda_1) \in \Ps^1$
\begin{equation} \label{eq-always-double}
  \mbox{ the curve } V(\lambda_0 \fgg +  \lambda_1 \hhh) \cap \Pi \mbox{ has a double point in } P_1. 
\end{equation}  

\noindent (\ref{p31-six}):
Assume that $X_H:=X\cap H$ is a hyperplane section containing $k=6$ triple points of $X$ and no plane, or $k\geq 7$ triple points of $X$.

We first show that $X_H$ is reduced:
Suppose $X_H$ has a non-reduced component then  defining polynomial of $X$ is of the form
$\ell g_1+g_2^2 g_3$. Then  $X$ would be singular along the curve $\ell=g_1=g_2=0$, a contradiction.

We now show that $X_H$ contains at most two planes:

If $X_H$ contains five planes then each triple point is contained on at least three planes, and each plane contains four of the triple points. Hence we can have at most $\lfloor 5\cdot 4/3 \rfloor=6$ triple points, contradicting $k\geq 7$.
If $X_H$ is the union of three planes $\Pi_i$ with $i\in \{1,2,3\}$ and a quadric surface $S$ then all $k\geq 7$ triple points are contained in $\Pi_1\cup \Pi_2\cup \Pi_3$. 
Since each plane contains four triple points, we have at most 5 points on more then one plane, and  at least  two of the $k$ points are smooth points of  $\Pi_1\cup \Pi_2\cup \Pi_3$. These points are double points of $S$. This implies that $S$ is reducible, a contradiction.

Hence $X_H$ has either three irreducible components, one of which is a plane, or at most two irreducible components.

If $X_H$ contains two planes $\Pi_1,\Pi_2$ then we claim that at most one of the triple points is on $\Pi_1\cap \Pi_2$. If there were two triple points then at most six of the triple points are contained in $\Pi_1\cup \Pi_2$. Hence the residual cubic $S$ has a triple point and therefore $S$ an irreducible cubic cone. The four triple points on $(\Pi_1\cup \Pi_2)\setminus (\Pi_1\cap \Pi_2)$ are double points of $S$. Since no three triple points are on a line we find that $S$ has four double lines and therefore is non-reduced. This contradicts the reducedness of $X_H$.

For the final step of the proof we renumber the points in such a way that $P_1,P_2,P_3,P_4$ are in general linear position. If $X_H$  contains a plane then we assume that the points are chosen such that $P_1,P_2,P_3$ and $P_7$ are coplanar. 
Then we have that $P_5$ and $P_6$ are not on any plane spanned by three points from $\{P_1,P_2,P_3,P_4\}$.

Let $\varphi:H \dashrightarrow H$ be the Cremona transformation centered at $P_1$, $P_2$, $P_3$ and $P_4$. Since $X_H$ has triple points at $P_1,P_2,P_3,P_4$ a direct calculation yields that $\varphi(X_H)$ has degree at most 3.

The rational map $\varphi$ contracts the planes spanned by any subset of three points from $\{P_1,P_2,P_3,P_4\}$ and is an immersion everywhere else.  In particular, $\varphi(X_H)$ has at most two irreducible components and each component is reduced. A cubic surface with at least two triple points is either non-reduced or is the union of three planes, hence $\varphi(X_H)$ has at most one triple point. On the other hand, $\varphi(P_5)$ and $\varphi(P_6)$ are distinct triple points of $\varphi(X_H)$, a contradiction.

\noindent (\ref{p31-five}):
Suppose $X$ contains three planes  $\Pi_1,\Pi_2,\Pi_3$ through a line $\ell$.
Without loss of generality we may assume that the three 2-planes  $\Pi_i$ on $X$  are $V(x_2,x_3)$, $V(x_2,x_4)$, $V(x_3,x_4)$.
Suppose for the moment that two of the triple points are on $\ell$. Then we may assume that   
\[ P_1=(1\colon 0\colon 0\colon 0\colon 0)\mbox { and }P_2=(0\colon 1\colon 0\colon 0\colon 0).\] 
A straightforward computation shows that every monomial of every generator of the ideal
\[ \mathrm{I}(P_1)^3 \cap \mathrm{I}(P_2)^3 \cap \mathrm{I}(\Pi_1) \cap \mathrm{I}(\Pi_2) \cap \mathrm{I}(\Pi_3)\]
of degree at most 5 is divisible by $x_0^a x_1^{2-a}$ for some $a$. In particular, one has $\ell \subset \sing(X)$. A contradiction.

If at most one of the triple points of $X$ is on $\ell$ then there are seven triple points contained in the hyperplane containing $\Pi_1$ and $\Pi_2$. This is excluded by (\ref{p31-six}).

\noindent (\ref{p31-seven}):
Let $S_2$ be given by $x_0=q_2(x_1,\dots,x_4)=0$. A quintic threefold $X$ containing $S_2$ is of the form
\[ x_0 h_0(x_0,\dots,x_4)+q_2h_2(x_1,\dots,x_4)=0.\]
The singular locus of $X$ contains the intersection $x_0=h_0=q_2=h_2=0$. This is a scheme of length 24. 

Let $P$ be a triple points of $X$. If $p$ is a singular point of both $q_2=0$ and $h_2=0$ then the tangent cone of $X$ at $p$ is reducible, contradicting the fact that $p$ is an ordinary triple point. Hence we can pick local coordinates $z_1,\dots,z_n$ such that $X$ is locally given by 
\[z_1g_1(z_1,z_2,z_3,z_4)+z_2g_2(z_2,z_3,z_4)=0\]
and we want to determine the length of $z_1=z_2=g_1=g_2=0$. The argument used in the case of a $2$-plane (see (\ref{p31-four})) can also be applied here and we find that the length of this scheme equals  4 at each triple point, hence there are exactly 6 triple points of $X$ 
on $S_2$.

\noindent (\ref{p31-eight}): The hyperplane section $X_H := X\cap H$ is reduced by the proof of part (\ref{p31-six}). 
If $X_H$  is reducible then it contains either a $2$-plane or a quadric surface. Thus, by (\ref{p31-seven}),  it contains a $2$-plane. \\
Suppose $X_H$ contains three $2$-planes $\Pi_1,\Pi_2,\Pi_3$. Then $\Pi_j \cap \sing(X)$ consists of four points by 
(\ref{p31-four}).
Moreover,  claims (\ref{p31-six}) and (\ref{p31-five}) yield that each pair of the  $2$-planes meets along a line that contains exactly two triple points.
Hence we may assume that $\Pi_1$ contains the points $P_1,P_2,P_3,P_4$ and $\Pi_2$ contains $P_1,P_2,P_5,P_6$. Then,
the quadruplet of triple points on 
the $2$-plane $\Pi_3$ splits into a pair of triple points that belong to  $\Pi_1$ and a pair of triple points  on $\Pi_2$.
 Hence $\Pi_3$ contains the points $\{P_3,P_4,P_5,P_6\}$ and these four points are on a 2-plane. If there were a fourth plane $\Pi_4 \subset X$ then by the same reasoning it would have to contain these four points and hence 
$\Pi_3=\Pi_4$, a contradiction.
\end{proof}

As a consequence of the proof of Proposition~\ref{prpBasic}.\ref{p31-eight} we obtain the following observation.

\begin{observation} \label{obs-useful}
{\rm (1)} If  the hyperplane section $H \cap X$ contains two  $2$-planes $\Pi_0$, $\Pi_1$, then $H$ contains  $6$ triple points of $X$ and exactly two singularities of $X$ belong to the line  $\Pi_0 \cap \Pi_1$. \\
{\rm (2)} Suppose   the hyperplane section $H \cap X$ contains three (distinct) $2$-planes $\Pi_0$, $\Pi_1$, $\Pi_2$. Then 
 the set of the six triple points of $X$ on $H$  splits into a pair of points in $\Pi_0 \cap \Pi_1$, a pair
in  $\Pi_0 \cap \Pi_2$ and a pair in  $\Pi_1 \cap \Pi_2$ .
\end{observation}

\section{Quintic threefolds containing a plane}\label{secPlane}

In this section we maintain the assumption [\Ajeden] from Section~\ref{secInv}, i.e., all singularities of a quintic $X \subset \Ps^4$ 
are ordinary triple points $P_1$, $\ldots$, $P_{t}$.   Moreover we make the additional assumption:

\begin{assumption}
{\rm {\bf [\Adwa]}}  the quintic threefold $X$ contains a $2$-plane $\Pi_0$.
\end{assumption}

For a hyperplane $H$ such that $\Pi_0 \subset H$ we define the surface
\begin{equation} \label{eq-def-qh}
Q := Q(H) := \overline{(X\cap H)\setminus \Pi_0}
\end{equation}
As we will show in the proof of Proposition~\ref{prop-fibers}.1, the surface  $Q$ is always a quartic.

The main results of this section is  the following proposition. 

\begin{proposition}\label{corPlanes}
Suppose $X$ is  a quintic threefold $X$ satisfying {\rm [\Ajeden], [\Adwa]}. If $X$ has at least 11 ordinary triple points, then there exists three distinct hyperplanes $H_1$, $H_2$, $H_3$ $\in$ $|{\mathcal O}(1) - \Pi_0|$ such that 
$Q(H_1)$ and  $Q(H_2)$ are both the union of two planes together with a quadric surface, whereas $Q(H_3)$   is either a plane together with a cubic surface with isolated singularities 
or again the union of two planes with a quadric surface.
\end{proposition}
The proof of Proposition~\ref{corPlanes} can be found at the end of this section.

By Proposition~\ref{prpBasic} we can label the triple points in such a way that $P_1,\ldots,P_4$ are the only triple points of $X$ on the plane $\Pi_0$. At first we study 
the behaviour of hyperplane sections around $P_1,\ldots,P_4$. 
\begin{proposition} \label{prop-fibers}
{\rm (1)} For general  $H \in |{\mathcal O}(1) - \Pi_0|$ the surface $Q$ is a quartic with  exactly four singular points   $P_1,\ldots,P_4$ and each
singularity of $Q$ is a node (i.e., an $A_1$-singularity).

\noindent
{\rm (2)} Fix $j \leq 4$. Then there exist precisely five hyperplanes $H \in |{\mathcal O}(1) - \Pi_0|$ such that  $P_j$ is no longer a singularity of type $A_1$ on $Q=Q(H)$. 
\end{proposition}
\begin{proof} (1) Let us  maintain the notation from the proof of parts (\ref{p31-three}), (\ref{p31-four}) of Proposition~\ref{prpBasic}. In particular  
$\Pi_0 = V(x_0,x_1)$,  and  $\FX$ can be written as \eqref{eq-expx0x1}.  
For $H = V(\lambda_0 x_0 - \lambda_1 x_1)$, the surface $Q(H)$ is given by $(\lambda_0 \fgg +\lambda_1 \hhh)$.
This shows that $Q$ is a quartic for every $H$. Moreover, by \eqref{eq-always-double}, $P_1$, $\ldots$, $P_4$ are double points on $Q$. 
By Bertini,  $Q(H)$ is smooth away from  $P_1$, $\ldots$, $P_4$ for general $H$. Finally, the quartic $Q$ is nodal for general $H \in |{\mathcal O}(1) - \Pi_0|$
by (2).

\noindent   
(2) Assume that $P_j = (0:\ldots:0:1)$ (i.e., $j=1$). Then the tangent cone $C_{P_1}X$ is a cone over a smooth cubic surface given by \eqref{eq-tang-cone}.
It contains $\Pi_0$ which is a cone over a line $\ell$ on the cubic  \eqref{eq-tang-cone}, so the hyperplane section $(C_{P_1}X) \cap H$  is a cone over a degree three curve which contains $\ell$.

Assume that  the conic residual to $\ell$ on a hyperplane section of the cubic surface  is irreducible. Then $C_{P_1}X \cap H$  splits into $\Pi_0$ and the cone over the conic.
Moreover, by \eqref{eq-always-double}, the intersection $\Pi_0 \cap Q$ has always a double point in $P_1$, so $\Pi_0$ cannot be a component of the tangent cone $C_{P_1}Q$.
Thus $C_{P_1}Q$ is a cone over a smooth quadric curve and 
$P_1$ is an $A_1$ singularity. 

Since there are 
10 lines on the cubic  \eqref{eq-tang-cone}  that meet the line $\ell$, there are precisely five hyperplanes for which the residual conic is reducible. Those hyperplanes yield the degeneration of the $A_1$ singularity.
\end{proof}

Let $X'$ be the blowup of $X$ along $\Pi_0$. Since $\Pi_0$ is a Cartier divisor on $X$ away from the triple points the blow-up is isomorphism away from the points $P_1,\ldots,P_4$, 
and each $P_j \in \Pi_0$ is replaced by a smooth rational curve on $X'$. In particular, from \eqref{eq-euler-char-singular} we obtain:
\begin{equation} \label{eq-euler-char-bl-plane-one}
\chi(X')=-196+16t.
\end{equation}
Let $\pi\colon X'\to \Ps^1$ be the morphism  induced on $X'$ by the projection from $\Pi_0$. One can easily see that its fibers are isomorphic to the quartic surfaces $Q(H)$, which were introduced in Proposition~\ref{prop-fibers}.
Thus, Proposition~\ref{prop-fibers} implies for  general  $P \in \Ps^1$, that we have $\chi(\pi^{-1}(P)) = 20$, and as in \cite[Proposition~III.11.4.ii]{bphv}
\begin{equation} \label{eq-euler-char-bl-plane-two}
\chi(X')=20 \chi(\Ps^1) + \sum_{P\in \Ps^1} (\chi(\pi^{-1}(P))-20)
\end{equation}

In order to state  Proposition~\ref{prop-epsilon} we introduce  the following notation.
\begin{notation}\label{notStd}
Assume that the quintic threefold $X$ satisfies [\Ajeden], [\Adwa], and  $\sing(X) \cap \Pi_0 = \{P_1,\dots,P_4\}$. Let  $H \in |{\mathcal O}(1) - \Pi_0|$  and let $Q=Q(H)$ be given by \eqref{eq-def-qh}. Then
\begin{itemize}

\item
$r(Q)$ is the number of points  in  $\sing(X) \cap (Q \setminus \Pi_0)$, 

\item
$s(Q)$ is the number of points $P_i \in \Pi_0$ such that $P_i$ is not a node of the quartic $Q$ 

\item 
$\epsilon(Q):=s(Q)+8r(Q)$.
\end{itemize}
\end{notation}
The proof of  Proposition~\ref{corPlanes} is based on careful study of Euler numbers of non-nodal quartic surfaces:
\begin{proposition} \label{prop-epsilon}
 Let  $H \in |{\mathcal O}(1) - \Pi_0|$. 
If $\chi(Q(H))+\epsilon(Q(H)) > 20$, then one of the following holds.
\begin{enumerate}
\item $Q$ is the union of an irreducible cubic surface and a plane. The cubic surface has 4 double points, two of them are also on the plane $\Pi_0$. We have $s(Q)=2$, $r(Q)=2$ and $\chi(Q)+\epsilon (Q) \leq 24$.
\item $Q$ is the union of two quadric cones, such that the vertex of each cone is contained on the other cone. In this case we have $\chi(Q) =  3$, $s(Q)=4$, $r(Q)=2$ and $\chi(Q)+\epsilon(Q) =  23$. 
\item $Q$ is the union of two planes and a quadric. In this case we have $\chi(Q) \leq 6$, $s(Q)=4$, $r(Q)=2$ and therefore  $\chi(Q)+\epsilon(Q) \leq 26$.
\end{enumerate}
\end{proposition}
For the sake of clarity of exposition  we give now the  proof of  Proposition~\ref{corPlanes}, and present the proof of
 Proposition~\ref{prop-epsilon} in  Section~\ref{secFib}.  

\begin{proof}[{Proof of Proposition~\ref{corPlanes}}]
Let $\Delta$ be the set of points $R \in \Ps^1$ such that $\pi^{-1}(R)$ is not (isomorphic to) a nodal quartic with exactly four singularities. 
For each  point $R \in \Delta$ we define the non-negative integer:
$$
\dedelta(R):= \max \{0, \chi(\pi^{-1}(R))+\epsilon(\pi^{-1}(R))-20 \}
$$
From Proposition~\ref{prop-fibers}.2 one obtains  $\sum_{R\in \Delta} \epsilon(\pi^{-1}(R))=8(t-4)+20$, so \eqref{eq-euler-char-bl-plane-two}
yields
\begin{equation} \label{eq-playing-with-numbers}
\chi(X')+20+8(t-4)=20\cdot 2 +\sum_{R\in \Delta} (\chi(\pi^{-1}(R))+\epsilon(\pi^{-1}(R))-20)
\end{equation}
Observe that the right hand side of \eqref{eq-playing-with-numbers} is bounded by $40+\sum_{R \in \Delta} \dedelta(R)$. From \eqref{eq-euler-char-bl-plane-one} we obtain the inequality
\[ 24t\leq 248 +\sum_{R \in \Delta} \dedelta(R)\]
 
If $t>10$ then $t=11$ by Lemma~\ref{lemSpec} and we have
\begin{equation} \label{eq-sixteen}
 16 \leq \sum_{R \in \Delta} \dedelta(R).
\end{equation}
By Proposition~\ref{prop-epsilon}, if  $\dedelta(R)$ is (strictly) positive, then  we have $r(\pi^{-1}(R)) =2$.  Since $t=11$, we can have at most three points $R \in \Delta$  with $\dedelta(R)>0$.

Moreover, Proposition~\ref{prop-epsilon} implies that  $\dedelta(R)\leq 6$, so there are   at least two fibers satisfying $\dedelta \geq 5$ and a third fiber satisfying $\dedelta\geq 4$.
\end{proof}

\section{Reduction to the Segre configuration}\label{secSegre}
It is well-known (see e.g. \cite{scub}, \cite[Section 3.2]{hunt}) that there is a unique  cubic threefold with 10 nodes (up to an automorphism of $\Ps^4(\C)$). 
The 10 nodes of such a cubic threefold form the so-called \emph{Segre configuration}. This configuration can be realised as the set of 10 such  points that each 
has exactly three coordinates equal to $1$ and the remaining two coordinates are equal to $(-1)$. 
However, to simplify our calculations  we will use a different  representation of the Segre  configuration in this section.

Below we show the following result:
\begin{proposition} \label{prop-segreconf}
Let $X \subset \Ps^4(\C)$ be a quintic threefold. Assume  that

{\rm [\Aoldtwo]} $X$ contains a $2$-plane $\Pi_0$,

{\rm [\Aoldone]} $\sing(X)$  consists of 11 ordinary triple points $P_1$, $\ldots$, $P_{11}$.

\noindent
Then a subset of $\sing(X)$ forms the Segre configuration.
\end{proposition}

The proof of Proposition~\ref{prop-segreconf} will be preceded by several  lemmata. It should be pointed out that the claim of Proposition~\ref{prop-segreconf}
does not hold for quintics with $10$ ordinary triple points, four of which lie on a $2$-plane (see Example~\ref{example-sevtwo},~\ref{example-sevthree}). 

Whenever we speak of  [\Aoldtwo],[\Aoldone] we mean the conditions stated in Proposition~\ref{prop-segreconf}.  
We say that $P_1, \ldots, P_4$ are {\sl coplanar} if and only if  $\langle P_1, \ldots, P_4 \rangle$ is two-dimensional. 
Whenever it leads to no ambiguity  
we say that a variety is {\sl contained in the fiber of a rational map} if and only if it is contained in its Zariski-closure.

At first we 
carry out a convenient change of coordinates.

\begin{lemma} \label{lem-coord}
Suppose that $X$ satisfies {\rm  [\Aoldtwo],[\Aoldone]}. Then we can  assume that there exist $a_1,\dots, a_4,b_3,b_4 \in \C$ such that the following equalities hold:
\begin{equation} \label{eq-10points}
\begin{matrix}
P_1 = (0\colon 0\colon 1\colon 0\colon 0), & & 
P_2 = (0\colon 0\colon 0\colon 1\colon 0), \\
P_3 = (0\colon 0\colon 0\colon 0\colon 1), & & 
P_4 = (0\colon 0\colon 1\colon 1\colon 1), \\
P_5  = (0\colon 1\colon 0\colon 0\colon 0), & & 
P_6 = (0\colon 1\colon 1\colon 1\colon 0), \\
P_7 = (1\colon 0\colon 0\colon 0\colon 0), & & 
P_8 = (1\colon 0\colon 1\colon 0\colon 1), \\
P_9 = (1\colon \aaa\colon \bb\colon \cc \colon \dd), & & 
P_{10} = (1\colon \aaa\colon \bb\colon \ee\colon \ff) .\\
\end{matrix}
\end{equation}
\end{lemma}
\begin{proof} Recall that $\Pi_0$ contains exactly  four triple points of $X$ by  Proposition~\ref{prpBasic}.\ref{p31-four}. We call them $P_1$, $\ldots$, $P_4$.

By   Proposition~\ref{corPlanes} there exist hyperplanes $H_1$, $H_2$, $H_3$  that contain the plane $\Pi_0$ and such that  each of $H_1 \cap X$, $H_2 \cap X$  splits  into three $2$-planes and a quadric,
whereas  $X\cap H_3$ consists either of three $2$-planes and a quadric, or two $2$-planes and a cubic. 

By Observation~\ref{obs-useful}  each $H_j$ contains exactly six singularities of $X$.
We put $P_5$, $P_6$ (resp.  $P_7$ and $P_8$, resp. $P_9$ and $P_{10}$) to denote  the triple points on the hyperplane $H_1$ (resp. $H_2$, resp. $H_3$) that do not belong to $\Pi_0$.

Let   $\Pi_1$, $\Pi_2$ be the $2$-planes residual to $\Pi_0$ in  $H_1 \cap X$. It follows from Observation~\ref{obs-useful}
that the $2$-planes $\Pi_1$, $\Pi_2$ meet along a line that 
contains exactly two singularities of $X$, i.e.,   both these planes contain $P_5$, $P_6$ and we can relabel the $P_i$'s, for $i=1,2,3,4$ such that 
\begin{equation} \label{eq-1256}
P_1,P_2,P_5, P_6  \in \Pi_1 \mbox{ and } P_3, P_4, P_5, P_6 \in \Pi_2 \, .
\end{equation}

Since $\Pi_0$ and $\Pi_1$ meet along the line $\langle P_1,P_2 \rangle$, 
Proposition~\ref{prpBasic}.\ref{p31-five} implies that the point $P_2$ does not belong to the plane  $\langle P_1,P_7,P_8 \rangle$. Hence, after swapping $P_3$, $P_4$ if necessary, Observation~\ref{obs-useful} allows us to  assume that 
\begin{equation} \label{eq-1378}
P_1,P_3,P_7,P_8 \mbox{ and } P_2,P_4,P_7,P_8 \mbox{ are coplanar.}
\end{equation}

Finally, one of the $2$-planes contained in $H_3 \cap X$ has to go through $P_9,P_{10}$ and either through $P_1,P_4$ or $P_2,P_3$ (this is Observation~\ref{obs-useful} combined with Proposition~\ref{prpBasic}.\ref{p31-five} again).
 We still may  swap $(P_1,P_3)\leftrightarrow (P_2,P_4)$. Hence we may  assume that
\begin{equation} \label{eq-23910}
P_2,P_3,P_9,P_{10} \mbox{ are coplanar.}
\end{equation}

By Proposition~\ref{prpBasic} the points $P_1,P_2,P_3,P_5$ and $P_7$ are in general linear position. 
Thus we may assume that $P_1,P_2,P_3,P_5$ and $P_7$ are as in \eqref{eq-10points}.

Moreover, since $P_4 \in \Pi_0$ we arrive at  $P_4=(0\colon 0\colon \alpha\colon \beta\colon \gamma)$, with $\alpha\beta\gamma\neq0$. After rescaling $x_3$, $x_4$ we obtain $\alpha=\beta=\gamma=1$, i.e., $P_4$ is as claimed.

From \eqref{eq-1256} we have $P_6 =(0\colon 1\colon \eta\colon \eta\colon 0)$ with $\eta\neq 0$. After rescaling $x_1$ we may assume that $\eta=1$. 
Furthermore, by \eqref{eq-1378}, we have  $P_8=(1\colon 0\colon \omega\colon 0\colon \omega)$ and after rescaling $x_0$ we assume that $\omega=1$. Thus $P_6$, $P_8$  are as as desired. 

Finally, we can find $\aaa$, $\bb$, $\cc$, $\dd$ such that $P_9=(1\colon \aaa\colon \bb\colon \cc\colon \dd)$. Then  \eqref{eq-23910} implies that $P_{10}=(1\colon \aaa\colon \bb\colon \ee\colon \ff)$ for some  $\ee,\ff \in \C$.
\end{proof}
In order to prove Proposition~\ref{prop-segreconf} (i.e., to find constraints on $\aaa$, $\bb$, $\cc$, $\dd$, $\ee$, $\ff$) 
we will use projections from various $2$-planes on $X$.
For this purpose, we collect some obvious consequences of Proposition~\ref{corPlanes} below. 
\begin{lemma}  \label{lem-fibers}
Assume {\rm [\Aoldone]}. Let $\Pi := \langle P_{i_1}$, $P_{i_2}$, $P_{i_3}$, $P_{i_4} \rangle$ be  a $2$-plane on $X$
and let  $\psi: (X \setminus \Pi) \rightarrow \Ps^1(\C)$ be the projection from $\Pi$. 
Then
\begin{enumerate}
\item exactly three fibers of $\psi$ contain a $2$-plane $\neq \Pi$ and each such fiber contains exactly two points from $\sing(X) \setminus \Pi$, 
\item if two different fibers of $\psi$ contain a $2$-plane $\neq \Pi$ each, then at least one of them contains two $2$-planes,  
\item if $\psi^{-1}(\psi(P_{i_5}))$ contains two $2$-planes, one which is $\langle P_{i_3}, P_{i_4}, P_{i_5}, P_{i_6} \rangle$, then 
$\psi^{-1}(\psi(P_{i_5}))$ contains the $2$-plane  $\langle P_{i_1}, P_{i_2}, P_{i_5}, P_{i_6} \rangle$. Hence
$$
P_{i_1}, P_{i_2}, P_{i_5}, P_{i_6}  \mbox{ are coplanar. }
$$
\end{enumerate}
\end{lemma}
\begin{proof} Each fiber of $\psi$ is given by a hyperplane section $H \cap X$. If $H \cap X$  contains $\Pi$ and another $2$-plane, then
it contains at least $6$ triple points of $X$ by Proposition~\ref{prpBasic}.\ref{p31-four} and Proposition~\ref{prpBasic}.\ref{p31-two}. Thus  Proposition~\ref{prpBasic}.\ref{p31-six}  shows that $H$ contains exactly $6$ triple points of $X$,
four of which are on $\Pi$.    In particular, each fiber with a $2$-plane $\neq \Pi$ contains exactly two points from $\sing(X) \setminus \Pi$.   

\noindent
(1) follows now from Proposition~\ref{corPlanes} and [\Aoldone].

\noindent
(2) is immediate consequence of (1) and Proposition~\ref{corPlanes}.

\noindent
(3) follows from  Observation~\ref{obs-useful}.
\end{proof}

\begin{lemma} \label{lemma-1579}
Suppose  that {\rm  [\Aoldtwo], [\Aoldone]} are satisfied and \eqref{eq-10points} holds. Then (after a change of coordinates if necessary) we can assume that 
\begin{equation} \label{eq-1579}
P_1,P_5,P_7,P_{9} \mbox{ are coplanar. }
\end{equation} 
while {\rm  [\Aoldtwo], [\Aoldone]} and  \eqref{eq-10points} remain valid.
\end{lemma}
\begin{proof} At first we claim that (after relabelling the points $P_1$, $\ldots$, $P_{10}$ if necessary)
either \eqref{eq-1579} holds or the points $P_2,P_5,P_7,P_{9}$ are coplanar
and none of the eight quadruplets of points 
\begin{eqnarray} \label{eq-eight-quadruplets}
&& \{P_1,P_5,P_7,P_9\}, \{P_1,P_6,P_7,P_9\}, \{P_1,P_5,P_8,P_9\}, \{P_1,P_6,P_8,P_9\} \\
&& \{P_1,P_5,P_7,P_{10}\}, \{P_1,P_6,P_7,P_{10}\}, \{P_1,P_5,P_8,P_{10}\}, \{P_1,P_6P_8,P_{10}\} \nonumber
\end{eqnarray}
are coplanar.

Indeed,
let $\psi_1(x_0, \ldots, x_4) := (x_0, x_4)$. Since  $\psi_1$ is  the projection from the  $2$-plane $\Pi_1 = \langle P_1,P_2,P_5,P_6 \rangle$,  
we can apply Lemma~\ref{lem-fibers} to study its fibers.

Observe that the fiber $\psi_1^{-1}(\psi_1(P_3))$  contains the $2$-planes $\langle P_1,P_2,P_3,P_4\rangle$, $\langle P_3,P_4,P_5,P_6 \rangle$. 
Thus, by  Lemma~\ref{lem-fibers}.1,  $\psi_1$ maps the remaining five points $P_7$, $\ldots$, $P_{11}$ to three points in $\Ps^1(\C)$. 

Obviously, we have $\psi_1(P_7) \neq \psi_1(P_{8})$. Moreover,  Proposition~\ref{prpBasic}.\ref{p31-two} shows that 
$P_2,P_9,P_{10}$ cannot be  collinear, so $\dd \neq \ff$ and  $\psi_1(P_9) \neq \psi_1(P_{10})$.
Hence one of the points $P_1,P_2$, one of $P_5,P_6$, one of $P_7,P_8$ and one of $P_9,P_{10}$ span a $2$-plane that is contained in a fiber of $\psi_1$ (recall Proposition~\ref{prpBasic}.\ref{p31-five}).

Our choices \eqref{eq-1256}, \eqref{eq-1378}, \eqref{eq-23910} do fix the points $P_1$, $P_2$ and the sets $\{P_l,P_{l+1}\}$ for $l=5,7,9$, but  
 we can  swap $P_5,P_6$ (resp. $P_7$, $P_8$, resp. $P_9$,$P_{10}$). This means, as we claimed at  the very beginning of the proof, that   
either \eqref{eq-1579} holds or the points $P_2,P_5,P_7,P_{9}$ are coplanar
and none of the eight quadruplets \eqref{eq-eight-quadruplets} are coplanar.

In the next step of the proof we  will show that if the points  $P_2,P_5,P_7,P_9$ are coplanar and 
none of the eight quadruplets \eqref{eq-eight-quadruplets} consists of coplanar points, then there exists an $\aaa \in \C$ such that
\begin{equation} \label{eq-strange-points}
P_9=(1\colon \aaa\colon 0\colon -1\colon 0) \mbox{ and } P_{10}=(1\colon \aaa\colon 0\colon 0\colon -\aaa).
\end{equation}

Suppose that $\langle P_2,P_5,P_7,P_9 \rangle$ is two-dimensional (i.e., we have $\bb=\dd=0$) and consider the projection 
$\psi_2(x_0, \ldots, x_4) := (x_2, x_4)$ from that plane.

By direct check, $\psi_2$ maps the six points $P_j$, where $j \neq 2,5,7,9, 11$,  to exactly three points in $\Ps^1(\C)$, so the fiber $\psi_2^{-1}(\psi_2(P_{11}))$ contains no $2$-planes. 

Suppose that the fiber  $\psi_2^{-1}(\psi_2(P_{1}))$
contains two $2$-planes. Since the fiber in question contains $\Pi_1$, Lemma~\ref{lem-fibers}.3 implies that  the points  $P_1,P_6,P_7,P_9$ are coplanar and we are done.

Thus, Proposition~\ref{corPlanes} allows us to assume the both  fibers $\psi_2^{-1}(\psi_2(P_{3}))$, $\psi_2^{-1}(\psi_2(P_{4}))$ contain two $2$-planes. Since  
both the $2$-planes  $\langle P_2,P_4,P_7,P_8\rangle$, $\langle P_2,P_3,P_9,P_{10}\rangle$ lie on $X$, the quadruplets  $P_4,P_5,P_8,P_9$ and $P_3,P_5,P_7,P_{10}$ 
consist of coplanar points by  Lemma~\ref{lem-fibers}.3. In particular, we have  $\cc=-1 \mbox{ and } \ee=0$.

Let $\psi_3$ be the projection from  the plane $\langle P_3,P_5,P_7,P_{10}\rangle$. Suppose that the fiber  $\psi_3^{-1}(\psi_3(P_{1}))$
contains two $2$-planes. By \eqref{eq-1378} and  Lemma~\ref{lem-fibers}.3  the points $P_1,P_5,P_8,P_{10}$ are coplanar and we are done. 
Thus we can assume that  $\psi_3^{-1}(\psi_3(P_{1}))$ contains exactly one plane.   Lemma~\ref{lem-fibers}.2 combined with \eqref{eq-1256} yields that 
the fiber $\psi_3^{-1}(\psi_3(P_{4}))$ contains two planes, one of which is $\Pi_2$. As before we use Lemma~\ref{lem-fibers}.3 to find the other $2$-plane in the fiber in question.
In particular, $P_4,P_6, P_7, P_{10}$ are coplanar, so we have $\aaa=-\ff$. Altogether, we have arrived at \eqref{eq-strange-points}.

To complete the proof assume \eqref{eq-strange-points} and consider the map
\begin{equation*} \label{eq-us-map}
\varphi(x_0\colon x_1\colon x_2\colon x_3\colon x_4) := (-x_1\colon x_0\colon x_0-x_2\colon x_0-x_2+x_3\colon x_4-x_2).
\end{equation*}
One can easily see that the set $\{P_1,\dots P_8\}$ is invariant under $\varphi$, whereas the points $P_1,P_5,P_7,\varphi(P_9)$ are coplanar, which completes the proof.
\end{proof}

\begin{lemma} \label{lemma-2679}
Suppose  that {\rm  [\Aoldtwo], [\Aoldone]} are satisfied and \eqref{eq-10points}, \eqref{eq-1579} hold. Then (after a change of coordinates if necessary) we can assume that 
\begin{equation} \label{eq-2679}
P_2,P_6,P_7,P_9 \mbox{ are coplanar }
\end{equation}
while  {\rm  [\Aoldtwo], [\Aoldone]} and  \eqref{eq-10points},  \eqref{eq-1579}  remain valid.
\end{lemma}

\begin{proof} 
Recall that by \eqref{eq-1579} we have  $\cc=\dd=0$ in \eqref{eq-10points}. 
Consider the projection $\psi_4(x_0\colon \dots\colon x_4) :=  (x_3\colon x_4) $ from the $2$-plane $\langle P_1,P_5,P_7,P_9 \rangle$.
Now  \eqref{eq-1579} allows us to use Lemma~\ref{lem-fibers}.3
to examine the fibers of $\psi_4$ in the same way we studied the fibers of $\psi_2$, $\psi_3$ in the proof of the previous lemma.

By \eqref{eq-1256} the fiber  $\psi_4^{-1}(\psi_4(P_{2}))$ contains the plane $\Pi_1$ and the extra triple point $P_6$. If $\psi_4^{-1}(\psi_4(P_{2}))$ contains two $2$-planes, then Lemma~\ref{lem-fibers}.3
yields  \eqref{eq-2679}.

Thus we can assume that $\psi_4^{-1}(\psi_4(P_{2}))$ contains exactly one plane. But \eqref{eq-1378} implies that   $\psi_4^{-1}(\psi_4(P_{3}))$
contains the plane $\langle P_1,P_3,P_7,P_8 \rangle$ and the triple point $P_8$. 
Lemma~\ref{lem-fibers}.2 yields that the fiber  $\psi_4^{-1}(\psi_4(P_{3}))$ contains two $2$-planes. From Lemma~\ref{lem-fibers}.3 we infer that
the points $P_3,P_5,P_8,P_9$ are coplanar (i.e.,  $\bb=1$). We obtain 
\begin{equation} \label{eq-strange-points2}
P_9=(1\colon \aaa\colon 1\colon 0\colon 0) \mbox{ and } P_{10}=(1\colon \aaa\colon 1\colon \ee\colon \ff).
\end{equation}

To complete the proof assume that \eqref{eq-strange-points2} holds and  consider the map $\varphi_2(x_0,x_1,x_2,x_3,x_4):=(x_1,x_0,x_2,x_4,x_3)$.
One can easily check that the  set $\{P_1,\dots,P_8\}$ is invariant under $\varphi_2$, whereas  
each of the quadruplets $P_1,P_5,P_7,\varphi_2(P_9)$, $P_2,P_6,P_7,\varphi_2(P_9)$  consist of coplanar points.
\end{proof}

In order to proceed further we have to exclude some degenerate configurations of the $10$ points given by  \eqref{eq-10points}. 
\begin{lemma} \label{lem-reducible}
Let  $P_1, \ldots, P_{10}$ be given by \eqref{eq-10points} with  $\aaa^5 \neq 1$ and 
\begin{equation} \label{eq-strange-points3}
P_9=(1:\aaa:\aaa:0:0), \quad P_{10}=(1:\aaa:\aaa:\aaa:1).
\end{equation}
 Then no quintic threefold with triple points in $P_1, \ldots, P_{10}$
is irreducible. 
\end{lemma}
\begin{proof}
Obviously, the quintic
\[  (x_2-x_3)x_3(x_1-x_2+x_4)(x_0-x_4)(x_0-\aaa x_1)\]
has multiplicity three in the ten points $P_{j}$, where $j \leq 10$. Computation with \cite{Sing} gives the basis of the vector space of quintics in 
$\cap_{i=1}^8 I(P_i)^3$.  Then one checks directly that the extra condition of vanishing of all  second derivatives in the points $P_9$, $P_{10}$ given by \eqref{eq-strange-points3} defines a unique degree $5$ threefold 
as soon as $\aaa^5 \neq 1$. 
\end{proof}

\begin{lemma} \label{lem-final}
Suppose  that {\rm  [\Aoldtwo], [\Aoldone]} are satisfied. Then  we can assume that \eqref{eq-10points} holds with
\begin{equation*} \label{eq-strange-points4}
P_9=(1:1:1:0:0), \quad P_{10}=(1:1:1:\ee:\ff).
\end{equation*}
\end{lemma}

\begin{proof} By Lemmata~\ref{lem-coord},~\ref{lemma-1579},~\ref{lemma-2679} we can assume that \eqref{eq-10points} holds with 
$$
P_9=(1:\aaa:\aaa:0:0), \quad P_{10}=(1:\aaa:\aaa:\ee:\ff).
$$
We claim that  $\aaa^5=1$.
Indeed,  we consider the projection $\psi_5$ from the $2$-plane $\langle P_2,P_6, P_7,P_9 \rangle \subset X$. By \eqref{eq-23910} (resp. \eqref{eq-1378})
its fiber over $\psi_5(P_{3})$ (resp. $\psi_5(P_{4})$) contains the $2$-plane $\langle P_2,P_3,P_9,P_{10} \rangle$ (resp. $\langle P_2,P_4,P_7,P_8 \rangle$).
By Lemma~\ref{lem-fibers}.3 and Lemma~\ref{lem-fibers}.2
either  $P_4,P_6,P_8,P_9$ or $P_3,P_6,P_7,P_{10}$ are  coplanar. The former yields $\aaa=1$, in which case we are done. Thus we can assume that the latter quadruplet consists of coplanar points and
the equality  $\aaa=\ee$ holds.

We consider the projection from the plane $\langle P_1,P_3,P_7,P_8 \rangle$. Its fibers contain the planes $\langle P_1,P_5,P_7,P_9 \rangle$, $\langle P_3,P_6,P_7,P_{10} \rangle$, so 
either $\langle P_3,P_5,P_8,P_9 \rangle$ or $\langle P_1,P_6,P_8,P_{10} \rangle$ is two-dimensional. In the first case we have $\aaa=1$ and in the second case we obtain $\ff=1$. Thus 
either $\aaa=1$ or \eqref{eq-strange-points3} holds. Lemma~\ref{lem-reducible} yields the equality $\aaa^5 =1$. 

Assume that $\aaa^5=1$, $\aaa\neq 1$ and  put $P_{11}:=(1:c_1:c_2:c_3:c_4)$.  We  consider the projection $\psi_4$ from the $2$-plane $\langle P_1,P_5,P_7,P_9 \rangle$ again. 
By direct check (see the proof of Lemma~\ref{lemma-2679}) we have   $\psi_4(P_2) = \psi_4(P_6)$,  $\psi_4(P_3) = \psi_4(P_8)$. 
Moreover,  Lemma~\ref{lem-fibers}.1 implies that the projection $\psi_4$, when restricted to the set of seven triple points away from the center $\langle P_1,P_5,P_7,P_9 \rangle$,
has exactly four fibers: three of them consist of two points and one of exactly one point. Therefore, from $\aaa\neq 1$ and
$$
\psi_4(\{P_4, P_{10},P_{11}\}) = \{ (1:1), (\aaa:1), (c_3:c_4) \} \, .
$$
we infer that either $(1:1)=(c_3:c_4)$ or $(\aaa:1)=(c_3:c_4)$.  Thus the following equality holds
\[ 
(c_3-c_4)(\aaa c_4 - c_3) = 0  \, .
\]

Repeating the above argument for the following $2$-planes $\langle P_1,P_6,P_8,P_{10}  \rangle$, $\langle P_2,P_3,P_9,P_{10} \rangle$, $\langle P_2,P_4,P_7,P_8 \rangle $ and $\langle P_3,P_4,P_5,P_6 \rangle$ we find the 
additional  quadrics:
\begin{eqnarray*}
(1-c_4-c_1+c_3)(1-c_4-\aaa c_1+\aaa c_3)&=& 0\\
(c_2-\aaa)((1-\aaa)(c_1-\aaa)+\aaa(c_2-\aaa))&=& 0\\
(c_2-c_4)((\aaa-1)c_1-\aaa(c_2-c_4))&=& 0\\
(1+c_3-c_2)(\aaa+c_3-c_2)&=& 0
\end{eqnarray*}
One can check with \cite{Sing} that the above system of five equations has no solutions whenever $\aaa^4-\aaa^3+\aaa^2-\aaa+1=0$.
This completes the proof.
\end{proof}

After those preparations we are in position to give the proof of Proposition~\ref{prop-segreconf}.

\begin{proof}[{Proof of Proposition~\ref{prop-segreconf}}]
By Lemma~\ref{lem-final} we can assume that \eqref{eq-10points}  holds with  $P_9=(1:1:1:0:0)$ and  $P_{10}=(1:1:1:\ee:\ff)$. 

We claim that $\ee=\ff=1$. Indeed, 
recall that  $\langle P_2,P_3,P_9,P_{10} \rangle$ is a $2$-plane on $X$ (see \eqref{eq-23910}) and put $\psi_6$ to denote the projection from that plane.
Since the $2$-planes $\langle P_1,P_2,P_3,P_4 \rangle$, $\langle P_2,P_6,P_7,P_9 \rangle$ and  $\langle P_3,P_5,P_8,P_9 \rangle$ are contained in fibers of  $\psi_6$,
two of the quadruplets $\{ P_1,P_4,P_9,P_{10} \}$, $\{P_3,P_6$,$P_7,P_{10}\}$ and $\{P_2,P_5$,$P_8,P_{10}\}$ are coplanar (this is  Lemma~\ref{lem-fibers} again). 
Thus two of the three equalities $\ee=\ff,\ee=1,\ff=1$ hold and we obtain 
\begin{equation} \label{eq-strange-segre}
P_9=(1:1:1:0:0) \quad P_{10}=(1:1:1:1:1)  \, .
\end{equation}

Finally, one can easily check that the cubic
\begin{equation} \label{eq-segrecubic}
\segcub := x_0 x_1 x_3 - x_0 x_1 x_4 - x_0 x_2 x_3 + x_0 x_3 x_4 + x_1 x_2 x_4 - x_1 x_3 x_4
\end{equation}
has double points at $P_1$, $\ldots$, $P_{10}$ given by \eqref{eq-10points}, \eqref{eq-strange-segre} and no further singularities.
Thus the points in question form the Segre configuration by \cite{scub} (see also \cite[Section 3.2]{hunt}). 
\end{proof}

\section{Proof of Theorem~\ref{thm-no-eleven}}\label{sec11}

Here we show that if a hyperplane section of a  quintic threefold $X$ is reducible then $X$ can have at most $10$ ordinary triple points (Theorem~\ref{thm-no-eleven}). Examples in Section~\ref{secExa}
show that  this bound is sharp. 

The proof of Theorem~\ref{thm-no-eleven} will be based on the following lemma.
\begin{lemma} \label{lem-no11}
Let $Y \subset \Ps^4(\C)$ be an irreducible  quintic threefold. If $\sing(Y)$ contains 11 triple points, 10 of which form the Segre configuration then $Y$ has a singularity that is not an    
ordinary triple point.
\end{lemma}

\begin{proof} Let $P_1$, $\ldots$, $P_{11}$ be isolated triple points  of the irreducible quintic $Y$. We assume \eqref{eq-10points}, \eqref{eq-strange-segre} and put 
\begin{equation} \label{eq-ideal-ten}
\mathfrak{I}_j := \cap_{i\leq 10} I(P_i)^j \mbox{ for } j=1,2,3.
\end{equation} 
 Moreover we consider the union of  five hyperplanes each of which contains exactly 6 points $P_j$, $j \leq 10$, given by  the degree $5$ polynomial  $G \in \mathfrak{I}_3$ :
$$
G:=(x_1-x_0)(x_2-x_4)(x_0-x_2+x_3)(x_3-x_1)x_4.
$$

Let $\segcub \in \mathfrak{I}_2$ be given by \eqref{eq-segrecubic}.  A  calculation with \cite{Sing} shows that the vector space of quintics in  $\mathfrak{I}_3$ (resp. quadrics in  $\mathfrak{I}_1$) is six-dimensional (resp. five-dimensional). Hence every degree $5$
irreducible polynomial $F \in \mathfrak{I}_3$ can be written as
$$
F = G_2 \segcub - G
$$
where $G_2 \in \mathfrak{I}_1$ is a quadric. 

Suppose now that $F$ is a generator of the ideal  of $Y$ and put $\ell_i := \langle P_i, P_{11} \rangle$ to denote  the line connecting $P_i$ and $P_{11}$ for $i \leq 10$. By Proposition~\ref{prpBasic}.\ref{p31-six} we have $G(P_{11})\neq 0$, so $\segcub(P_{11}) \neq 0$. Moreover,
since $\ell_i  \subset Y$, we have  $$ (G_2 \segcub)|_{\ell_i}=G|_{\ell_i}.$$ 
Parametrize $\ell_i$ in such a way that the parameter $t=0$ (resp. $t=\infty$) corresponds to the point $P_i$ (resp. $P_{11}$).
Then the restriction $G|_{\ell_i}$ (resp. $\segcub|_{\ell_i}$)
is a degree $5$ (resp.  degree $3$)  polynomial with a triple root (resp. double root)  at $t=0$. Hence the non-zero root of $\segcub|_{\ell_i}$ is also a root of the restriction $G|_{\ell_i}$.

Put $P_{11}:=(c_0:c_1:c_2:c_3:c_4)$. By direct computation the non-zero  root of  $\segcub|_{\ell_1}$ 
is $(c_0 c_3 - c_1 c_4)/\segcub(P_{11})$. Substitution of that root into $(G|_{\ell_1})/t^3$ yields the product
$$
c_0 c_1 c_3 c_4 (c_0-c_1) (c_1 - c_3)(c_3 - c_4)(c_4-c_0)(c_0-c_1+c_3-c_4) 
$$
up to the square of $\segcub(P_{11})$ in the denominator.
This implies that $P_{11}$ lies on a hyperplane which contains already six of the triple points $P_1$ $\ldots$, $P_{10}$. 
Proposition~\ref{prpBasic}.\ref{p31-six} completes the proof.
\end{proof}

\begin{proof}[{Proof of Theorem~\ref{thm-no-eleven}}]
The number of points in $\sing(X)$ cannot exceed $11$ by Lemma~\ref{lemSpec}.
Moreover, by  Proposition~\ref{prpBasic}.\ref{p31-eight}, the quintic $X$ contains a $2$-plane.
If $\sing(X)$ consists of $11$ points,  Proposition~\ref{prop-segreconf} and Lemma~\ref{lem-no11} yield a contradiction.
\end{proof}

\begin{remark} \label{remark-strongerclaim} 
We conjecture that  as soon as $X$ contains a complete intersection surface of bidegree $(a,b)$ with $a,b <5$, then it also contains a $2$-plane, so the bound of Theorem~\ref{thm-no-eleven}
holds even under weaker assumptions.
Indeed, if $X$ contains a complete intersection surface, then it contains a complete intersection surfaces of bidegree $(1,1)$, $(1,2)$ or $(2,2)$. In Proposition~\ref{prpBasic} we show that if it contains a surface of bidegree $(1,2)$ then it contains one of bidegree $(1,1)$. Suppose now that $X=V(F)$ contains a complete intersection surface of bidegree $(2,2)$. Then we can write $F$ as
$g_1h_1+g_2h_2$
with $\deg(g_i)=2$ and $\deg(h_i)=3$. Thus, to show the existence of a $2$-plane on $X$ 
one needs a detailed knowledge of 
the geometry of certain cubic threefolds. We do not follow this path to maintain our exposition compact.    
\end{remark}

\section{Examples}\label{secExa}

Below we discuss various examples of quintic threefolds with ordinary triple points. It should be pointed out that the quintic $V(F_{vS})$ from 
Example~\ref{example-duco} appears
 in \cite[page~862]{DucoQuintic} as a quintic with ten Del Pezzo nodes, but without calculation of the defect $\delta$. 
For ten points $P_1,\dots P_{10}$ we define  
$\mathfrak{I}_j$ by  \eqref{eq-ideal-ten}. The vector space of quintics in the  ideal $\mathfrak{I}_3$ is denoted by $\mathfrak{I}_{(5)}$.

\begin{example} \label{example-duco}
Let $P_1,\dots P_{10}$ form the Segre configuration given by  \eqref{eq-10points}, \eqref{eq-strange-segre}. 
It is well-known (\cite[Section 3.2]{hunt}) that this configuration is invariant under an action of the symmetric group  $S_6$ and $(\mathfrak{I}_{(5)})^{S_6}$ is one-dimensional (up to a change of coordinates, the set of zeroes of the basis vector  $F_{vS}$ is the irreducible quintic threefold ${\mathcal M}_{(-3:1)}$ from \cite[page~862]{DucoQuintic}). As we checked in the proof of Lemma~\ref{lem-no11}, the vector space $\mathfrak{I}_{(5)}$ is six-dimensional
and every irreducible $\FX \in \mathfrak{I}_{(5)}$  can be written as 
$$
\FX = \lambda F_{vS} + G_2 \segcub ,
$$ 
where $\lambda \neq 0$ is a constant,  $G_2 \in \mathfrak{I}_1$ is a quadric and $\segcub \in \mathfrak{I}_2$ is the Segre cubic (see \eqref{eq-segrecubic}). 

There are 15 quadruplets of coplanar points in $\{P_1,\dots,P_{10}\}$ and each such subset yields a plane $\Pi_j \subset V(\FX)$ by Proposition~\ref{prpBasic}. 
In this way we find 15 planes $\Pi_1$, $\ldots$, $\Pi_{15}$ in $X:=V(\FX)$. Using Proposition~\ref{prpDefectFormula} one shows that $\delta=14$.

These planes generate a subgroup of finite index of $\CH^1(X)/\Pic(X)$: The fifteen planes satisfy $\sum_{j=1}^{15} \Pi_j \in |{\mathcal O}_{X}(3)|$. 
To show that this is the only relation, we use a general hyperplane section $X_H$ of the $X$. Then $X_H$ is smooth. The fifteen planes yield fifteen lines. The intersection numbers of these lines can be easily determined. One checks that the Gram matrix of fourteen of these lines and the hyperplane class has nonzero determinant.
Hence these 14 lines are independent in $\CH^{1}(X_H)/\Z{\mathcal{O}_{X_H}(1)}$. Since the restriction map
\[ \CH^1(X)/\Z \mathcal{O}_X(1) \to \CH^{1}(X_H)/\Z{\mathcal{O}_{X_H}(1)}\]
is a group homomorphism, we find that the 14 planes are independent in $\CH^1(X)/\Z \mathcal{O}_X(1)=\CH^1(X)/\Pic(X)$.
\end{example}

\begin{example} \label{example-sevtwo}
We let $P_1,\dots,P_8$  be given by \eqref{eq-10points}, and put 
$$
P_9:=(1:a:a:0:0) \mbox{ and } P_{10}=(1:a:a:a:a) \mbox{ with } a\neq 0,1.
$$
Then,  $\{P_1,\dots,P_{10}\}$ contains exactly 11 quadruplets of coplanar points. 

By a direct computation with  \cite{Sing}  we have \[\dim(\mathfrak{I}_{(5)}) =2  \mbox{ for general } a \in \C\]
and a general member of $\mathfrak{I}_{(5)}$ has ten ordinary triple points.

The 11 quadruplets of coplanar points yield 11 $2$-planes on $X$.
One can easily calculate the intersection matrix of the 11 induced lines on a general hyperplane section $X_H$ of $X$  to we find that the 11 lines are independent in $\CH^{1}(X_H)/\Z{\mathcal{O}_{X_H}(1)}$.  As in the previous example we infer that the 11 planes are independent in $\CH^1(X)/\Z \mathcal{O}_X(1)$ and therefore $\delta\geq 11$.

Finally, we calculate the ideal \eqref{eq-def-ideal} and obtain $h_J(5)=99$, which shows that 
$\delta=11$ in this case and that the 11  planes generat a finite index subgroup of $\CH^1(X)/\Pic(X)$.
\end{example}

\begin{example} \label{example-sevthree}
Let $P_1,\dots,P_8$ be given by \eqref{eq-10points}, and put 
$$
P_9:=(1:1:a:0:0) \mbox{ and } P_{10}:=(1:1:a:a:a) \mbox{ with } a\neq 0,1.
$$
Now $\{P_1,\dots,P_{10}\}$ contains exactly  9 quadruplets of coplanar points. By hyperplane section argument, the resulting $2$-planes are independent in the group  $\CH^1(X)/\Z {\mathcal O}_{X}(1)$ and $\delta\geq 9$
for any quintic threefold in $\mathfrak{I}_{(5)}$ that satisfies [\Ajeden]. As in the previous example one can check that
$$
\dim(\mathfrak{I}_{(5)}) =1  \mbox{ for general } a \in \C.
$$  

A direct computation with  \cite{Sing} shows that for $a=7$  the unique generator of $\mathfrak{I}_{(5)}$ does define a quintic with 10 ordinary triple points. In this case $\delta=10$, which is consistent with the fact that 
$X$ seems to move in a one-parameter family given by the choice of a general  $a \in \C$, so $h^{2,1}(X) \geq 1$ and $\delta \geq 10$ (recall Proposition~\ref{prpDefectFormula}).  

It seems that in this case the positivity  of defect can be explained by the existence of the $2$-planes, but the group $\mathrm{CH}^1(X)/\mathrm{Pic}(X)$ cannot be generated only by $2$-planes.
\end{example}
Below we sketch  constructions of  a quintic with $9$ ordinary triple points.
\begin{example}
(1) Recall that, by \cite{SemRot} (see also \cite[$\S.2$]{SBarron}) 
every cubic threefold with $9$ nodes
up to automorphism of $\Ps^4$  is of the form
\[ x_0x_1x_2-x_3x_4(a_0x_0+a_1x_1+a_2x_2+a_3x_3+a_4x_4) \mbox{ where } a_1, \ldots, a_4 \in \C
\]
and it contains nine $2$-planes (see e.g. \cite[p.~197]{Finkel}). Moreover, each $2$-plane on such a cubic contains four nodes of $X$. 

Let $P_1,\dots,P_9$ be the nine double points of a 9-nodal cubic given by the parameters $a_1$, $\ldots$, $a_4$ and let $\mathfrak{I}_j := \cap_{i\leq 9} I(P_i)^j$. 
We claim that  $\mathfrak{I}_2$ contains at least two independent forms of degree 3. 
Indeed, one can easily check that both 
 $x_0x_1x_2$ and $x_3x_4(a_0x_0+a_1x_1+a_2x_2+a_3x_3+a_4x_4)$ belong to $\mathfrak{I}_2$.
We choose two linearly independent  cubics $\Cppp, \Cbbb \in \mathfrak{I}_2$, two  linearly independent degree-2 forms  $Q', Q'' \in \mathfrak{I}_2$
(observe that the latter exist because $\dim(H^0(\mathcal{O}_{\Ps^4}(2))=15$) and consider the quintic $X$ given by 
\begin{equation} \label{eq-extwotwo}
\Cppp Q' +  \Cbbb  Q''.
\end{equation}
By construction $X $ has nine triple points $P_1, \dots,$ $P_9$ and contains nine $2$-planes (by Proposition~\ref{prpBasic}.\ref{p31-four}). 
Direct computation  with \cite{Sing} shows that there is a choice of all parameters in the above construction such that 
$X$ has exactly nine singularities that are ordinary triple points (i.e., $\mu(P_i) = 16$ for $i=1, \ldots, 9$)
As in Example~\ref{example-sevtwo} the hyperplane section argument shows that the $9$  planes are independent in $\CH^1(X)/\Z\mathcal{O}_X(1)$ and we have  $\delta(X) \geq 9$.
A computation with  \cite{Sing} gives an example with $\delta(X) = 9$.

\noindent
(2) Let  $R_1$, $\ldots$, $R_8$ be nodes of an   8-nodal cubic threefold $\Cppp$ and let $\Cbbb$ be another 8-nodal cubic threefold with $\sing(\Cbbb) = \{R_1, \ldots$, $R_8 \}$.
We choose a general point $R_9 \in \Cppp \cup  \Cbbb$ and two linearly independent quadrics $Q', Q''$ that contain the eight nodes and have a double point in $R_9$.
Then the quintic given by (\ref{eq-extwotwo}) has nine triple points  $R_j$, $j =1, \ldots 9$. Direct computation  with \cite{Sing} shows that after
correct choice of cubics and quadrics the above construction does yield a quintic with $9$ ordinary triple points. Observe that  $\Cppp$ contains
exactly  five quadruplets of coplanar nodes
(see e.g. \cite[p.~196]{Finkel}).
Thus the quintic $X$ contains exactly five $2$-planes by Proposition~\ref{prpBasic}.\ref{p31-four}.   
\end{example}
Finally we discuss possible values of $\delta(X)$ for quintic threefolds $X$ with $t \leq 6$ ordinary triple points. We put    $\mathfrak{I}_3 := \cap_{i\leq t} \mbox{I}(P_i)^3$. 
Observe that by  Proposition~\ref{prpDefectFormula}  the following inequality holds
\begin{equation} \label{eq-boundfordelta}
\delta(X) \leq  11t - h_{\mathfrak{I}_3}(5).
\end{equation}

\begin{remark} \label{rem-fewpoints}
(1) Let $t=5$. Proposition~\ref{prpBasic}.\ref{p31-four} implies that $\sing(X)$ contains at most one quadruplet of coplanar points.
By direct computation (with help of (\ref{eq-boundfordelta}))  $\delta(X) = 1$ if and only if $\sing(X)$ contains  a quadruplet of coplanar points. 
Otherwise we have  $\delta(X) = 0$.

\noindent  
(2) Let  $t=6$.  If the set of six points is in general linear position (resp. contains precisely one quadruplet of coplanar points),  then the set is unique up to an automorphism.  
With help of (\ref{eq-boundfordelta}) one easily checks that  $\delta(X) = 0$ (resp. $\delta(X) = 1$).  \\
The final possibility is that there are at least two quadruplets of coplanar triple points. Then we relabel singularities so that both $P_1,P_2,P_3,P_4$ and $P_1,P_2,P_5,P_6$ are coplanar. 
Thus after applying an appropriate automorphism of $\Ps^4$ we have 
\begin{equation} 
\begin{matrix}
P_1 = (1\colon 0\colon 0\colon 0\colon 0),  & & 
P_2 = (0\colon 1\colon 0\colon 0\colon 0), \\
P_3 = (0\colon 0\colon 1\colon 0\colon 0),   & & 
P_4 = (1\colon 1\colon 1\colon 0\colon 0),\\
P_5  = (0\colon 1\colon 0\colon 1\colon 0),  & &  
P_6 = (1\colon a\colon 0\colon 1\colon 0). \\
\end{matrix}
\end{equation}
Since two $2$-planes are contained in $X$ we find that $\delta(X) \geq 2$. For general $a \in \C$ we have  $\delta(X) = 2$. However, if $a=1$ then we find that $P_3,P_4,P_5,P_6$ are coplanar and $\delta(X)=3$.
\end{remark}
For $t=7,8$ there are too many possibilities to obtain a simple classification of possible values of $\delta(X)$. This can be illlustrated by the following example.
\begin{example}
One can check that for 8 points in general position there is a quintic with isolated triple points at each of these points. 
The 8 points form an eight-dimensional family and one easily checks that for a general member the vector space of degree-5 forms in  $\mathfrak{I}_3$ is six-dimensional. This implies that $h^{2,1}(\tilde{X})\geq 13$. To show equality we can take 8 general points and calculate $\delta$ using Proposition~\ref{prpDefectFormula}. We find that  $\delta(X)=0$ and by Proposition~\ref{prop-hodgetildeX} that $h^{2,1}(\tilde{X})=13$.

On the other hand, let us assume that $P_1,\dots,P_8$ are  given by \eqref{eq-10points}. By direct check, the above  eight points contain exactly five quadruplets of coplanar points.  
A computation with  \cite{Sing} shows that there exists a quintic $X$ given by $\FX \in \mathfrak{I}_3$ with exactly eight ordinary triple points $P_i$, and 
$\delta(X)=5$.
\end{example}

\section{Fibers of the map $\pi$}\label{secFib}

In this section we assume that $X$  satisfies [\Ajeden], [\Adwa] and prove Proposition~\ref{prop-epsilon}.
For obvious reasons we only use results from Section~\ref{secBas} and Appendix~\ref{appPen}.

We maintain  Notation~\ref{notStd} (see page~\pageref{notStd}) and put $\Sigma$ to denote  the union of one-dimensional components  of $(\sing(Q))_{\red}$.
Observe that by Bertini we have  $\deg(\Sigma) \leq 3$. \\
Whenever we say that $Q$ is ruled, we mean that it is ruled by lines (c.f. \cite{TopQua}). 

\begin{lemma} Suppose $Q$ has isolated singularities then $\chi(Q)+\epsilon(Q)\leq 20$.
\end{lemma}

\begin{proof}
Since $\Sigma = \emptyset$ we have $\chi(Q)=24-\sum_{R\in \sing(Q)} \mu(R)$, where $\mu(R)$ is the Milnor number of the singular  point $R \in Q$. If $R$ is a triple point of $X$ and $Q\not \in \Pi_0$ then $R$ is a triple point of $Q$ and $\mu(R)\geq 8$. If $R \in \sing(X) \cap \Pi_0$ is not a node on $Q$, then $\mu(R) \geq 2$. Hence $\chi(Q)\leq 20-8r-s$.
\end{proof}

\subsection{$Q$ irreducible}
Assume now that $Q$ is irreducible and that $Q$ has non-isolated singularities. 

\begin{observation}\label{obs-ruledirred} If $Q$ is irreducible and ruled,  then $\Sigma$ contains a (possibly reducible) reduced conic $K$ such that $(Q \cap \Pi_0)_{\red} = K$.
\end{observation}
\begin{proof}
By \cite[Proposition 1.8(2)]{TopQua}  
we have $\sing(Q) = \Sigma$. By construction $\Sigma$ contains the coplanar points 
$P_1$, $\ldots$, $P_4$. From Proposition~\ref{prpBasic}.\ref{p31-two} the curve  $\Sigma$ contains a (possibly reducible) conic $K \subset \Pi_0$.
Obviously, $K$ comes up with mutiplicity at least two in the quartic  $Q \cap \Pi_0$.
\end{proof}

\begin{remark} \label{rem-notriple-line}
Observe that $Q$ cannot contain a  line of triple points. Indeed, the quartic $Q$ would be ruled then and the equality  $\Sigma = \ell$ would hold.
This is impossible by Observation~\ref{obs-ruledirred}. 
\end{remark}

After those preparations we can list possibilities for $\Sigma$.
\begin{proposition}\label{prpClass} If $Q$ is irreducible and $\Sigma  \neq \emptyset$ then one of the following occurs:
\begin{enumerate}
\item \label{quartic1}      $\Sigma$ is a line.
\item  \label{quartic2}     $\Sigma$ is a (possibly reducible)  conic.
\item  \label{quartic4}     $\Sigma$ is the union of three lines $\ell_1,\ell_2,\ell$ such that $\ell_1 \cap \ell = \emptyset$, $(Q \cap \Pi_0)_{\red} = (\ell_1 \cup \ell_2)$. In this case $Q$ is ruled.
\item  \label{quartic5}     $\Sigma$ is the union of a line $\ell$ and an irreducible conic $K$. In this case $K \cap \ell \neq \emptyset$, $Q$ is ruled and $(Q \cap \Pi_0)_{\red} = K$.
\end{enumerate}
\end{proposition}

\begin{proof}
Suppose that $\Sigma$ is a planar curve. Then $\deg(\Sigma) \leq 2$ (recall that $Q$ is irreducible). This is (\ref{quartic1}), (\ref{quartic2}).

If $\Sigma$ consists of two skew lines, then $Q$ is ruled  by  \cite[Lemma 4.3]{GonRams}. This is impossible by Observation~\ref{obs-ruledirred}.    

Assume that  $\deg(\Sigma)=3$. If $\Sigma$ is irreducible, then $\Sigma$ is a twisted cubic and $Q$ is ruled by \cite[Lemma 4.3]{GonRams}. This is impossible by Observation~\ref{obs-ruledirred}. \\
Suppose that $\Sigma$ consists of lines  $\ell_1,\ell_2,\ell$. Then two of them are skew (say $\ell_1 \cap \ell = \emptyset$), so $Q$ is ruled  by  \cite[Lemma 4.3]{GonRams}. 
By Observation~\ref{obs-ruledirred} we can assume that $\Pi_0 \cap Q  = (\ell_1 \cup \ell_2)$. This yields $\ell \cap \ell_2 \neq \emptyset $ and we have   (\ref{quartic4}).   \\
Finally, let $\Sigma =  \ell \cup K$, where $K$ is  an irreducible conic and let $\Pi$ be a general plane through the line $\ell$.
Since $K$ and $\ell$ meet in at most one point, the conic residual to $2\ell$ in $Q\cap \Pi$ has a double point (in one of the two points in $K\cap \Pi$), so it is reducible.
Thus $Q$ is ruled. By Observation~\ref{obs-ruledirred} we have (\ref{quartic5}).
\end{proof}

\begin{lemma} \label{lemm-Roneach}
Suppose $Q$ is irreducible, $\Sigma\neq \emptyset$ and $Q$ contains a triple point $R$ of $X$ which is not on $\Pi_0$. Then $R$ is on each irreducible component of $\Sigma$ and $\Sigma$ consists of a line or two intersecting lines. 
\end{lemma}

\begin{proof} Let $C\subset \Sigma$ be an irreducible component not containing $R$. Then any line $\ell$ through  $R$ that meets $C$  intersects $Q$ with multiplicity at least 5, and is therefore contained in $Q$. In particular, the cone over $C$ with vertex $R$ is a component of $Q$. Since $\deg(C)\leq 2$ this implies that the quartic $Q$ contains a surface of degree at most 2, contradicting the assumption that $Q$ is irreducible. Hence $R$ is on every irreducible component of $\Sigma$, which rules out the cases  (\ref{quartic4}), (\ref{quartic5}) of Proposition~\ref{prpClass}. 

Assume that the claim of  Proposition~\ref{prpClass}.\ref{quartic2} holds. Let  $\Pi$ be the plane that contains the conic $\Sigma$. 
Observe that  every plane section of $Q$ passing through $R$ has a point of  multiplicity at least 3. In particular,  $\Pi \cap Q$ is a double conic having a triple point at $R$, i.e., the conic $\Sigma$  is reducible.
\end{proof}

\begin{lemma}\label{lemPosTriple} If $Q$ is irreducible and  ruled,  then $r(Q)=0$ and $s(Q)=4$. \end{lemma}

\begin{proof}
Observation~\ref{obs-ruledirred} shows that none of the points $P_1$, $\ldots$, $P_4$ are  isolated singularities of $Q$ and we have  $s(Q) = 4$. 
Moreover $\Pi_0$ contains a component of $\Sigma$, so $Q$ contains no triple points from  $\sing(X) \setminus \Pi_0$ by  
Lemma~\ref{lemm-Roneach}.
\end{proof}

As a result we obtain  Proposition~\ref{prop-epsilon} in the case $Q$ is ruled.
\begin{proposition}\label{prpIrrRuled} If $Q$ is irreducible and ruled, then $\chi(Q)+\epsilon(Q)\leq 20$.
\end{proposition}
\begin{proof}

Assume that $\deg(\Sigma) \leq 2$. If $\Sigma$ is an irreducible conic then $Q$ is nonruled by  \cite[Lemma 3.1]{TopQua}. Thus the claim follows from Proposition~\ref{prpLine}
(see Appendix~\ref{appPen}).

It remains to consider the cases  (\ref{quartic4}), (\ref{quartic5}) of Proposition~\ref{prpClass}.
Moreover, by Lemma~\ref{lemPosTriple}  we have $\epsilon(Q)=4$, so it suffices to show that $\chi(Q) \leq 16$.
Consider now the pencil of 
 conics $C_{\Pi} \in  |{\mathcal O}_Q(1) - 2\ell|$. In the case  (\ref{quartic5}), if the plane $\Pi$ is not tangent to the conic $K$ in the point $K \cap \ell$ 
(resp. in the case (\ref{quartic4}) if  $\Pi$ does not contain the line  $\ell_2$),
then the conic $C_{\Pi} := Q \cap \Pi - 2\ell$ meets $K$ in a point away from $\ell$, so it is singular and therefore reducible. This implies that the generic fiber of the map $Q\setminus \ell \to \Ps^1$ is the union of two (affine)  lines (by generic smoothness) and degenerates in a double line, or a line together with $\ell$. This means that the Euler characteristic of the general fiber and of the special fiber is $1$. In particular $\chi(Q\setminus \ell)=2$ and $\chi(Q)=4$. 
\end{proof}

Now we can deal with the non-ruled case. 
\begin{proposition} \label{prpIrrNonRuled} Suppose $Q$ is irreducible, $\Sigma \neq \emptyset$ and $Q$ is not ruled. Then $\chi(Q)+\epsilon(Q) \leq 20$.

\end{proposition}
\begin{proof} By Proposition~\ref{prpClass} the curve  $\Sigma$ is planar and $\deg(\Sigma) \leq 2$.
Moreover, if $\Sigma$ contains a line, the claim follows from Proposition~\ref{prpLine}
(see Appendix~\ref{appPen}).

Suppose now that  $\Sigma$ is an irreducible conic $K$.
Then by \cite[Theorem 8.6.4]{DolCAG} we have that $Q$ is the projection of quartic surface in $\Ps^4$ from a point outside this surface. 
In this case we have a birational morphism $S\to Q$, which  induces  an isomorphism $S\setminus C \cong Q\setminus K$, for some quartic curve $C\subset S$.

Lemma~\ref{lemm-Roneach} yields  $r=0$ and the equality $\epsilon(Q)=s(Q)$.
 The surface $Q$ has at least $4-s$ double points away from $K$ and hence $S$ has at least $4-s$ double points. Since $C$ is a quartic curve we find $\chi(C)\geq -4$ and hence $\chi(S)\leq 10-(4-s)=6+s$. Now $\chi(Q)=\chi(S)-\chi(C)+\chi(K)\leq 6+s+4+2=12+s$. Since $\epsilon=s$ we find that $\chi(Q)+\epsilon \leq 12+2s\leq 20$. 
\end{proof}

\subsection{$Q$ is reducible}

Suppose $Q$ is reducible. By Proposition~\ref{prpBasic} the quartic $Q$ is reduced and has at most three irreducible components. Hence $Q$ is one of the following
\begin{enumerate}
\item the union of a plane and a cubic surface,
\item the union of two planes and an irreducible quadric,
\item the union of two quadrics.
\end{enumerate}
Below we analyze all cases in a series of lemmata. 
\begin{lemma}\label{lemCubSur} Suppose $Q(H)$ is the union of a plane $\Pi$ and a cubic surface $S$. Then $S$ has isolated singularities and $\chi(Q)+\epsilon(Q) \leq 24$.
\end{lemma}
\begin{proof} By Observation~\ref{obs-useful} the hyperplane $H$ contains six triple points of $X$, so we can assume that  the triple points $P_1,P_2,P_5,P_6$ belong to $\Pi$.
Thus the cubic surface $S$ has four singular points in $P_j$,  $j \geq 3$. Moreover, we have $P_1$, $P_2$ $\in$ $S$.  If $S$ had  nonisolated singularities then $\sing(S)$ would be  a double line, hence this is impossible by Proposition~\ref{prpBasic}.\ref{p31-two}. Thus  $S$ has four isolated singularities. By \cite{Bruce} the points $P_3$, $\ldots$, $P_6$  are nodes on $S$  and we have $\chi(S) = 5$. 

The intersection $S\cap \Pi$ is a cubic curve with double points $P_5$, $P_6$, so it contains the line $\langle P_5, P_6 \rangle$. The residual conic contains the four triple points of $X$, so  
 the cubic curve is either the union of three non concurrent lines or the  line  $\langle P_5, P_6 \rangle$  and an (irreducible) conic. In the first case $\chi(S\cap \Pi) = 3$ and 
in the second case we have $\chi(S\cap \Pi) = 2$ (observe that the line is not tangent to the conic).
We obtain $\chi(Q)=\chi(S)+\chi(\Pi)-\chi(S\cap \Pi) \in \{5, 6\}$.

Finally the surface $Q$ contains  6 triple points of $X$, so  $r(Q)=2$. Moreover, we have $P_1$, $P_2$ $ \in  \Pi \cap S$, whereas  $P_3$, $P_4$ are $A_1$-points on $Q$. 
Therefore $s(Q)=2$ and $\epsilon(Q) =18$, which yields $\chi(Q)+\epsilon(Q)  \in \{23, 24\}$.
\end{proof}

\begin{lemma}\label{lemTwoPlanes} Suppose $Q(H)$ is the union of two planes $\Pi_1,\Pi_2$ and a quadric surface $S$.
Then $\chi(Q)+\epsilon(Q) \leq 26$.
\end{lemma}

\begin{proof}
Since $X \cap H$ contains the three $2$-planes  $\Pi_0$, $\Pi_1,\Pi_2$, the residual quadric surface $S$ is irreducible by Proposition~\ref{prpBasic}.\ref{p31-eight}. Moreover,  by Observation~\ref{obs-useful},  we can assume 
that the triple points $P_1,P_2$,$P_5,P_6$ belong to $\Pi_1$ and $P_3$,$P_4$,$P_5$,$P_6$ are on the $2$-plane $\Pi_2$.
Hence $\Pi_0\cup \Pi_1\cup \Pi_2$ has a double point at each $P_i$,  and  $P_i \in S$ for  $i \leq 6$.

Suppose that the quadric $S$ does not contain the line $\langle P_5, P_6 \rangle$. 
Obviously,  the union $\Pi_1\cup \Pi_2$ has Euler characteristic $3+3-2=4$. 
Let $K_i = S \cap \Pi_i$ for $i=1,2$. Recall that both conics are reduced. By assumption $K_1 \cap K_2$ consists of the two points $P_5, P_6$, so
$$
\chi(S \cap (\Pi_1\cup \Pi_2)) = \chi(K_1) + \chi(K_2)  - 2.
$$
Accordingly, we obtain $\chi(Q) =     6 + \chi(S) - (\chi(K_1) + \chi(K_2)) \in \{3, 4,5,6\}$ depending on the number of singular conics $K_i$, because $\chi(S) \in \{ 3,4 \}$. \\
If $S$ contains $\langle P_5, P_6 \rangle$, then  $\chi(S \cap (\Pi_1\cup \Pi_2)) = 4$ and  $\chi(Q) = \chi(S) \in \{3,4\}$. 

In both cases we have 2 points in $Q \cap (\sing(Q) \setminus \Pi_0)$, so $r(Q)=2$. Since each $P_1$, $\ldots$, $P_4$  lies on the intersection of $S$ with a plane $\Pi_i$, $ i \neq 0$, we get $s(Q)=4$ and $\epsilon(Q) =20$.
Altogether we arrive at $\chi(Q)+\epsilon(Q) \in \{23, 24, 25,  26\}$ as claimed.
\end{proof}

\begin{lemma} \label{lemQuaSur} If $Q(H)$ is the union of two irreducible quadrics $Q_1$, $Q_2$, then we have $\chi(Q)+\epsilon(Q) = 23$.
\end{lemma}
\begin{proof}
By Proposition~\ref{prpBasic}  the hyperplane $H$ (resp. the quadric  $Q_1$, resp. $Q_2$) contains six triple points of $X$. After relabelling the points,  we can assume that $Q(H)$ has triple points $P_1$, $\ldots$, $P_6$.
Since none of the points $P_5$, $P_6$ are on $\Pi_0$, each of them is double point of one of the quadrics $Q_i$.  Thus we can assume that $Q_1$ (resp. $Q_2$) is a  cone with vertex  $P_5$ (resp. $P_6$).

Let $K_i := Q_i \cap \Pi_0$. We claim that the conic $K_i$ is smooth for $i=1,2$. Indeed, we have $P_1,\dots, P_4 \in K_i$. Thus $K_i$ cannot be union of two lines by  Proposition~\ref{prpBasic}.\ref{p31-two}
(recall that $P_{4+i}$ is the vertex of $Q_i$). 

We claim that $K_1=K_2$ and the cones $Q_1$, $Q_2$ touch along the line $\langle P_5, P_6 \rangle$. 
Indeed, let $R$ be the point in  $\langle P_5, P_6 \rangle \cap \Pi_0$. By Proposition~\ref{prpBasic}.\ref{p31-two} the point $R$ is none of the $P_i$'s.
On the other hand, we have   
$\langle P_5, P_6 \rangle \subset Q_1 \cap Q_2$, so $R \in K_1 \cap K_2$. Consequently 
the smooth conics $K_1$, $K_2$  have five intersection points and they coincide. Thus in every point on  $\langle P_5, P_6 \rangle$ 
the tangent spaces of both cones coincide.

Finally, we have two quadric cones ($\chi(Q_1) + \chi(Q_2)=6$) intersecting in the smooth conic $K_1$ and the  double line $2 \langle P_5, P_6 \rangle$.
Since   $K_1$ meets $ \langle P_5, P_6 \rangle$ in exactly one point, we have  $\chi(Q_1 \cap Q_2)=3$ and  $\chi(Q)=3$.  Moreover, $Q(H)$ is singular along the conic $K_1$, so 
$s(Q) = 4$. Since $r(Q) = 2$  we have $\epsilon(Q) =  20$ in this case and the proof is complete.
\end{proof}

\subsection{Proof of Proposition~\ref{prop-epsilon}}
After those preparations we can prove Proposition~\ref{prop-epsilon}.
\begin{proof}[Proof of Proposition~\ref{prop-epsilon}]
If  $\chi(Q)+\epsilon(Q) > 20$, then  $Q$ is reducible by Proposition~\ref{prpIrrRuled} and Proposition~\ref{prpIrrNonRuled}. 
Proposition~\ref{prpBasic}.\ref{p31-eight} implies that we are in one of the cases covered by 
Lemmata~\ref{lemCubSur},~\ref{lemTwoPlanes} and~\ref{lemQuaSur}. This finishes the proof.
\end{proof}

\appendix

\section{Pencils of conics on quartic surfaces with a double line}\label{appPen}
In this section we assume that  $Q\subset \Ps^3$ is  an irreducible  quartic surface and put  $\Sigma$ to denote the union of $1$-dimensional components of the singular locus $\sing(Q)$. Moreover, we   
assume that
\begin{equation} \label{eq-ass} 
\Sigma \mbox{ contains a line } \ell \mbox{ of double points.}
\end{equation}

In the sequel the following lemma will play crucial role.
\begin{lemma} \label{lem-atworstA3}
If  $Q$ has two singular points that are not coplanar with the line $\ell \subset \Sigma$,
then  the transversal type of $Q$ along the line $\ell$  is $A_k$ with $k \leq 3$.
\end{lemma}
\begin{proof}
We can assume that  $\ell := V(x_0,x_1)$ is a line of double points on $Q$  and $(1:0:0:0)$, $(0:1:0:0)$  are singularities of $Q$.
 Then the quartic $Q$ is given by the vanishing of 
\begin{equation} \label{eq-dlquartic}
f_{2,0} x_0^2+f_{1,1} x_0x_1+f_{0,2}x_1^2+ f_{2,1} x_0^2x_1+f_{1,2}x_0x_1^2+f_{2,2}x_0^2x_1^2
\end{equation}
with $f_{i,j}\in \C[x_2,x_3]$ of degree $(4-i-j)$, and $f_{2,0}, f_{0,2} \neq 0$.

Suppose that the transversal type of $Q$ along $\ell$ is not $A_1$. Then 
the form $(\sum_{i+j=2} f_{i,j} x_0^i x_1^j)$
is a square, i.e. there exists a non-empty open set $U \subset \ell$ (open in the analytic topology) and $g_0, g_1 \in {\mathcal O}(U)$ such that 
\[ \sum_{i+j=2} f_{i,j} x_0^i x_1^j =  (g_0 x_0+g_1 x_1)^2 . \]
Blowing-up along $V(x_0,x_1)$ one can easily check that 
the transversal singularity is not of type $A_2$ if and only if 
\begin{equation} \label{eq-notA2-constraint}
f_{2,1} g_1 - f_{1,2} g_0 = 0 \, .
\end{equation}

Assume  that the transversal type of $Q$ along $\ell$ is $A_k$ with $k \geq 3$. Thus over an open subset of $\ell$ 
the generator  \eqref{eq-dlquartic} can be written as
\[  (g_0 x_0+g_1 x_1)^2 + x_0 x_1 g_2 (g_0 x_0+g_1 x_1) + f_{2,2}x_0^2x_1^2.\]
Let $\tilde{Q}$ be the strict transform of $Q$ and let $\sigma:\tilde{Q}\to Q$ be the blow-up map. A direct calculation shows that 
for an open subset $V \subset \ell$  the map  $\sigma|_{\sigma^{-1}(V)}:\sigma^{-1}(V)\to V$ is bijective and that at a general point of $\sigma^{-1}(\ell)$ we have a double point. The hessian of a transversal slice at a generic point
  vanishes if and only if  $(4 f_{2,2} - g_2^2)$ vanishes. This shows that  
if the transversal singularity is worse then $A_3$, then $f_{2,2} = g_2^2/4$ and \eqref{eq-dlquartic} becomes a square, which yields a  contradiction.
\end{proof}

The blow-up of $Q$ along the line $\ell=V(x_0, x_1)$ can be see as  a subset of $\Ps^3 \times \Ps^1$. If we consider the homogenous coordinates $(x_0: \ldots : x_3)$ (resp. $(y_0:y_1)$) on 
$\Ps^3$ (resp. on $\Ps^1$), then its exceptional divisor  $E_1$ is the support of the $(2,2)$-divisor on the quadric $\ell \times \Ps^1$ given by the vanishing of  
\begin{equation} \label{eq-exdiv}
 f_{2,0}(x_2,x_3) y_0^2+f_{1,1}(x_2,x_3) y_0y_1+f_{0,2}(x_2,x_3) y_1^2
\end{equation}
If $Q$ has transversal A$_1$ singularity along $\ell$  and $E_1$ is irreducible then $\chi(E_1) \in \{0,1, 2\}$ (the latter two possibilities occur when $E_1$ is singular). Otherwise 
one can easily see that either $E_1$ is union of two (different) $(1,1)$-divisors (this yields  $\chi(E_1) \in \{2,3\}$)
or $E_1$ splits into the union $(1,1) + (1,0) + (0,1)$. In the latter case we have $\chi(E_1) \in \{ 3,4\}$. Altogether, for $k=1$
we have $\chi(E_1) \in \{0, 1, 2, 3, 4\}$.   

Suppose that the transversal type of $Q$ along $\ell$ is A$_k$ with $k>1$. In this case the $(2,2)$-divisor $E_1$ is non-reduced.
If $\mbox{supp}(E_1)$ is  irreducible, then  $E_1 =  2(1,1)$ and \eqref{eq-exdiv} is square of an irreducible polynomial. \\
The other possibility is $E_1 := 2(1,0) + 2(0,1)$. This implies that \eqref{eq-exdiv} can be written as square  
of product of two linear forms \\
Finally, we can have $E_1 := 2(1,0) + (0,1)+(0,1)$. 
In this case  \eqref{eq-exdiv} 
splits into three factors. One can easily check that in all cases we have  
\begin{equation} \label{eq-pos-chi}
\chi(E_1) \geq 0
\end{equation}

In the lemmata below we maintain the assumption that $Q$ is irreducible. 
We put  $E_1$ to denote the exceptional divisor of the blow-up of $Q$ along $\ell$.
The family of the planes that contain the line $\ell$
is denoted by  $\{ W_{t} \, : \,  t \in \Ps^1 \}$, whereas 
$C_t$ stands for the (possibly reduced or non-reduced) conic residual to $2\ell$ in the hyperplane section $(W_{t}  \cap Q)$.  

\begin{lemma} \label{lem-atworstA3-bis}
Let  $Q$ be  of the transversal type A$_3$ along the line $\ell$ and let  the linear system $|{\mathcal O}_Q(1) - 2\ell|$ contain  a smooth conic. 
Moreover, assume that  $\sing(Q)$  contains 
two points  that are not coplanar with the line $\ell$.

\noindent
{\rm (1)}  If  $\mbox{supp}(E_1)$  has three components, then $Q$ has exactly two singular points away from the line $\ell$.

\noindent
{\rm (2)} If  $\mbox{supp}(E_1)$  has at most two components, then
the pencil $|{\mathcal O}_Q(1) - 2\ell|$
 contains exactly three singular conics, all of which are double lines. Moreover, every smooth member of the pencil is tangent to the line
$\ell$ and one of the three double conics is $2\ell$.
\end{lemma}
\begin{proof} We maintain the notation of the proof of Lemma~\ref{lem-atworstA3}. In particular  $\ell := V(x_0,x_1)$ and $Q$ is given by  \eqref{eq-dlquartic}.

\noindent
(1) We can assume that 
\[ \sum_{i+j=2} f_{i,j} x_0^i x_1^j =  (x_0 + x_1)^2 (a'_2 x_2 + a'_3 x_3) (a''_2 x_2 + a''_3 x_3) \]
with  $a'_i, a''_i, b_j \in \C$. Then, \eqref{eq-notA2-constraint} implies the equality  $f_{2,1} = f_{1,2}$. 
Obviously, the substitution $x_0=tx_1$ in \eqref{eq-dlquartic}) and dividing out the square $x_1^2$ yields equations of degree-$2$ divisors $C_t$ in  the pencil 
 $|{\mathcal O}_Q(1) - 2\ell|$. A direct computation shows that the pencil contains exactly three singular quadrics: for $t=-1$ we obtain $2\ell$, whereas hyperplane sections
$V(x_0)\cap Q$, $V(x_1) \cap Q$ consist of unions of three distinct lines (one of them is obviously $\ell$). In particular,  the quartic $Q$ has at most two singular points away from the line $\ell$.    

\noindent
(2) Suppose that $E_1$ is irreducible. Then we  can assume that the form $\sum_{i+j=2} f_{i,j} x_0^i x_1^j$ is the square $(x_0 x_2 + x_1 x_3)^2$. Then
\eqref{eq-notA2-constraint} implies that $Q$ is given by 
\[ (x_0 x_2 + x_1 x_3)^2 + a x_0 x_1 (x_0 x_2 + x_1 x_3) + f_{2,2} x_0^2 x_1^2 .\]
A direct computation shows that every member of the pencil  $|{\mathcal O}_Q(1) - 2\ell|$
on the resulting quartic is reducible. Moreover, one can check that the quartic in question is singular along the lines
$V(x_0, x_1)$, $V(x_1, x_2)$, $V(x_0, x_3)$.

Since the assumption that the pencil contains an irreducible conic rules out the above case
we can assume that  $E_1$  has exactly two components and 
\[ \sum_{i+j=2} f_{i,j} x_0^i x_1^j =  (x_0 + x_1)^2 (a_2 x_2 + a_3 x_3)^2 \]
As in the proof of part (a) we obtain the equality  $f_{2,1} = f_{1,2}$ and put $x_0=tx_1$ in \eqref{eq-dlquartic}. 
Again the pencil contains exactly three singular quadrics: for $t=-1$ we obtain $2\ell$, whereas hyperplane sections
$V(x_0)\cap Q$, $V(x_1) \cap Q$ contain double lines $\neq 2 \ell$.  \\
Finally, by putting $x_1 =0$ in the equation of $C_t$ one checks that
every smooth member of the pencil meets the line $\ell$  in exactly one point. 
\end{proof}

\begin{lemma}\label{lemSingLine}
Let $Q \subset \Ps^3$ be a quartic surface, such that $\Sigma$ is a line $\ell$. 
\noindent
{\rm (1)} If a residual conic $C_{t_0}$ is smooth, then 
 for at most  $8$ planes $W_{t}$  (counted with multiplicity),  the residual conic $C_{t}$ is singular.

\noindent
{\rm (2)} Assume a residual conic  $C_{t_0}$ to be smooth. 
If $P$ is a singularity of the surface $Q$ that belongs to $W_{t_1} \setminus  \ell$, then the plane $W_{t_1}$ appears with multiplicity at least two among the eight planes that exist by (1).
\end{lemma}

\begin{proof}
Let $\ell = V(x_0,x_1)$ and let $W_t := V(x_0-tx_1)$.  Then the quartic $Q$ is given by the vanishing of 
\begin{equation} \label{eq-hypsec}
x_0^2 f_2(x_0,x_1,x_2,x_3)+ x_0x_1g_2(x_0,x_1,x_2,x_3)+x_1^2h_2(x_1,x_2,x_3)
\end{equation}
and the substitution $x_0=tx_1$ yields a generator of the ideal of  $C_{t}$:
\begin{equation} \label{eq-resCt}
 t^2 f_2(tx_1,x_1,x_2,x_3)+tg_2(tx_1,x_1,x_2,x_3)+h_2(x_1,x_2,x_3) \, .
\end{equation}

\noindent
(1)  We can assume that the conic $C_{\infty} \in Q\cap V(x_1)$ is smooth.  

\noindent
Obviously, 
the conic $C_t$ in $W_t \cong \Ps^2$ is given by  $(x_1,x_2,x_3) A(t) (x_1,x_2,x_3)^T=0$, where  
 $A(t) = [a_{i,j}(t)]_{i,j=1,\ldots, 3}$ is  a symmetric $(3\times 3)$-matrix with entries in $\C[t]$.
By \eqref{eq-resCt} the degree of an entry of $A(t)$ cannot exceed the corresponding entry of the following matrix
\[ \left( \begin{matrix} 
4 & 3 & 3 \\
3 & 2 & 2 \\
3 & 2 & 2
\end{matrix} \right)\]
Hence we have $\deg_t(\det(A(t))) \leq 8$ and, by assumption, $\det(A(t))$ is a non-zero polynomial.  
This implies that at most $8$ conics  $C_t$ are singular.

\noindent
(2) We assume that  $P:=(0:1:0:0) \in W_0$ is a singularity of the quartic $Q$. 
Then \eqref{eq-hypsec} implies $\deg_{x_1}(g_2)\leq 1$. Moreover, we have 
$\deg_{x_1}(h_2)=0$. In particular, the coefficient of $x_1^ax_2^bx_3^c$ in the polynomial \eqref{eq-resCt}
is divisible by $t^{a}$. Thus $t^2 | a_{1,1}(t)$ and $t | a_{1,j}(t)$ for $j = 2,3$. Since $A(t)$ is symmetric the 
determinant $\det(A(t))$  is divisible by $t^2$ and the proof is complete.
\end{proof}

\begin{remark} \label{rem-ap2} By generic smoothness,
if $\Sigma$ is a line, then
the residual conic $C_{t}$ is smooth away from $\ell$ for general $t$. Hence $C_t$ is either the union of two lines intersecting on $\ell$ or an irreducible conic. 
If $Q$ has an isolated singularity, then by \cite[Proposition 1.8(2)]{TopQua},  the quartic $Q$ is not ruled (by lines).  Thus the general $C_t$ is smooth
and the assumptions of Lemma~\ref{lemSingLine} are fulfilled.
\end{remark}

In the next lemma we apply similar approach to the quartics that are singular along two coplanar  lines.
\begin{lemma}\label{lemSingTwoLines} Let $\Sigma = \ell \cup \ell_1$, where $\ell$, $\ell_1$ are coplanar lines, and let $\Sigma \subset W_{t_0}$.

\noindent
{\rm (1)}  If a residual conic $C_{t_0}$ is smooth, then
at most four (counted with multiplicity) conics $C_{t}$, with $t \neq t_0$,    are singular.

\noindent
{\rm (2)}  Assume a residual conic $C_{t_1}$ to be smooth.
If $P$ is a singularity of the surface $Q$ that  belongs to $W_{t_2} \setminus  \Sigma$, then the plane $W_{t_2}$ appears with multiplicity at least two among the four planes that exist by (1).

\noindent
{\rm (3)} If a residual conic $C_{t_1}$ meets the line $\ell$ in two (distinct) points, then 
 at most two conics $C_{t}$ (counted with multiplicity), with $ t \neq t_0$,  are tangent to the line  $\ell$.
\end{lemma}

\begin{proof} Let $\ell$, $W_t$, $C_t$  and $A(t)$ be as in the proof of Lemma~\ref{lemSingLine} and let $\ell_1 = V(x_1, x_2)$.   
Then $Q$ is given by
\[ x_1^2h_2(x_0, x_1, x_2, x_3) +x_0x_1x_2g_1(x_0,x_2,x_3)+(x_0x_2)^2f_0\]
and $\Sigma \subset W_{\infty} = V(x_1)$. As in  \eqref{eq-resCt}, the substitution $x_0=tx_1$ yields a generator of the ideal of  $C_{t}$:
\begin{equation} \label{eq-resCt2} 
h_2(x_1,tx_1,x_2,x_3)+tx_2g_1(tx_1,x_2,x_3)+t^2x_2^2f_0 \, .
\end{equation}

\noindent
{\rm (1)} Consider the matrix $A(t) = [a_{i,j}(t)]_{i,j=1,\ldots, 3}$ that defines $C_t$. 
One can easily check that   $\deg_t(a_{i,j}(t))$ cannot exceed the corresponding entry of the following matrix
\[ \left( \begin{matrix} 
2 & 2 & 1 \\
2 & 2 & 1 \\
1 & 1& 0\\ 
\end{matrix} \right)\]
In particular,  we have $\det(A(t)) \neq 0$ and   $\deg_t(\det(A(t))) \leq 4$, so the claim follows. 

\noindent
{\rm (2)} Assume that  $P:=(0:1:0:0) \in W_0$ is a singularity of $Q$, observe that $\deg_{x_1}(h_2)=0$ then,  and repeat almost verbatim the proof of Lemma~\ref{lemSingLine}.2.

\noindent
{\rm (3)}  Substitute $x_1=0$ in the generator \eqref{eq-resCt2} to obtain the quadratic polynomial $h_2(0,0,x_2,x_3)+t x_2 g_1(0,x_2,x_3)+t^2 x_2^2 f_2$.
Observe that its  discriminant $\in \C[t]$  has degree at most 2 (c.f. proof of Lemma~\ref{lemSingLine}.3).
 \end{proof}

In the lemma below we maintain the assumption \eqref{eq-ass}.
\begin{lemma} \label{lem-ell-on-pi}
Let the line $\ell$ be the only $1$-dimensional component of $\sing(Q)$. If the quartic surface $Q$ has two isolated singular points on a plane $W$ that contains the line $\ell$,  then   $Q$ is of transversal type A$_1$ along $\ell$. 
\end{lemma}
\begin{proof} We can assume that $\ell = V(x_0, x_1)$, $W = V(x_0)$,  and the point $(0:1:0:0)$ is a singularity of $Q$. Thus 
\begin{equation} \label{eq-exp-on-pi}
f = \sum_{i+j \geq 2}   f_{i,j} x_0^i x_1^j  \mbox{ where }  f_{i,j} \in \C[x_2,x_3] \mbox{ and } \deg(f_{i,j}) = 4-(i+j)  
\end{equation}
Moreover, from  $(0:1:0:0) \in \sing(Q)$, we infer that    $f_{0,3}$, $f_{1,3}$ and $f_{0,4}$ vanish.
Finally, 
the conic residual to $2 \ell$ in $W \cap Q$ has two singularities away from $\ell$, so it is a double line $\neq \ell$.
By putting $x_0=0$ into \eqref{eq-exp-on-pi}  we obtain $x_1^2 f_{0,2}$, so  $f_{0,2}$ is square of a linear form, i.e. 
$f_{0,2} = f_{0,1}^2$, and $2V(x_0, f_{0,1})$ is the conic residual to $2 \ell$ in the hyperplane section $W \cap Q$.

Suppose that the transversal type of $Q$ along $\ell$ is not A$_1$ . This yields that on an open subset $U \subset \ell$ we have (c.f. \eqref{eq-ass}) 
$$
\sum_{i+j = 2}   f_{i,j} x_0^i x_1^j  = (\tilde{f}_{1,0} x_0 + \tilde{f}_{0,1} x_1)^2  
$$
with some $\tilde{f}_{1,0}$, $\tilde{f}_{0,1} \in {\mathcal O}(U)$. In particular, 
the coefficient $f_{1,1} \in \C[x_2, x_3]$ vanishes along the line $V(x_0, f_{0,1})$.
Then, by direct computation we have  
$$
\partial f/\partial x_0 |_{V(x_0, f_{0,1})}  =  f_{1,2} x_1^2 
$$
whereas the other partials (i.e. $\partial f/\partial x_j$  for $j \geq 1$) 
vanish along the line in question. By assumption, the quartic $Q$ has two singularities on $V(x_0, f_{0,1})$ away from $\ell$,
so the degree-$1$ polynomial $f_{1,2}$ has at least two roots on  $V(x_0, f_{0,1})$. Thus all partials of \eqref{eq-exp-on-pi}
vanish along $V(x_0, f_{0,1})$ and $\Sigma \neq \ell$. Contradiction. 
\end{proof}

After those preparations we can finally give a proof of Proposition~\ref{prpLine}.
It should be emphasized that in the proof we only use results from Section~\ref{secBas},  Appendix~\ref{appPen}
and Lemma~\ref{lemm-Roneach}, so Proposition~\ref{prpLine} can be used to prove Propositions~\ref{prpIrrRuled},~\ref{prpIrrNonRuled}.

\begin{proposition}\label{prpLine}
Let $X$  be a quintic threefold that satisfies the assumptions {\rm [A0]}, {\rm [A1]} and let 
$Q$  $\in$  $|{\mathcal O}_X(1) - \Pi_0|$
be a quartic surface.
If $\Sigma$ is either a line of double points or a union of two coplanar lines of double points,  then 
\begin{equation} \label{eq-maininequality}
\chi(Q)+\epsilon\leq 20.
\end{equation}
\end{proposition}

\begin{proof}
Let $\ell \subset \Sigma$ be a line   and let  $Q$ be of transversal type $A_k$ along $\ell$.

At first we assume  that  $\Sigma = \ell$. By Remark~\ref{rem-ap2} the pencil  $|{\mathcal O}_Q(1) - 2\ell|$ always contains  a smooth conic.

 If $\ell \subset \Pi_0$, then by Proposition~\ref{prpBasic}  the line $\ell$ contains at most two triple points of $X$, so $Q$ has two singularities on $\Pi_0 \setminus \ell$.
 Thus  Lemma~\ref{lem-ell-on-pi} yields $k =1$ and Lemma~\ref{lemm-Roneach} gives $r=0$. 
Otherwise,   $\ell$ runs through at most one triple point of the quintic threefold on $\Pi_0$,  so $Q$ has at least three double points on $\Pi_0$ away from $\ell$ 
and Lemma~\ref{lem-atworstA3} implies  $k\leq 3$.

At first we exclude the case $r=2$. Let $P_5,P_6 \in \sing(X)$ be triple points of $Q$. By Lemma~\ref{lemm-Roneach} we have $P_5$, $P_6 \in \ell$, so $P_j \notin \ell$ for $j \leq 4$ (see Proposition~\ref{prpBasic}.\ref{p31-two}). \\
We can assume that $\ell=V(x_0,x_1)$ , the quartic $Q$ is given by \eqref{eq-dlquartic}, and $P_5=(1:0:0:0)$, $P_6=(0:1:0:0)$. By direct check this implies that $f_{i,j} = a_{i,j} x_2 x_3$ for $i+j=2$
and  $((V(x_0-tx_1) \cap Q) - 2\ell)|_{\ell}$ is given by $((a_{2,0}t^2+a_{1,1}t+a_{0,2}) \cdot x_2 x_3)$. 
Thus the line $\ell$ is always a component of a singular conic in the pencil $|{\mathcal O}_Q(1) - 2\ell|$. 
On the other hand, the  points $P_1,\dots,P_4$ do not belong  to $\ell$, and  no line $\langle P_i, P_j \rangle$, where $i,j \leq 4$, is coplanar with $\ell$ by Proposition~\ref{prpBasic}.\ref{p31-four},
so the linear system  $|{\mathcal O}_Q(1) - 2\ell|$  contains four singular conics, each of which is singular at one of the points $P_1, \ldots, P_4$.
 Altogether we found $5$ singular members of  $|{\mathcal O}_Q(1) - 2\ell|$,
which is impossible by Lemma~\ref{lemSingLine}. This contradiction shows that $r \leq 1$.

We  assume $k \leq 2$, consider a general quartic $G\in \mbox{I}(\ell)^2$ and put  $Q_g:=V(G)$. Let $\tilde{{Q}}_g$ be the blow-up of $Q_g$ along the line $\ell$.
Since $k \leq 2$,  
the surface $\tilde{{Q}}_g$ has only isolated singularities.  
Moreover,  Bertini implies that the quartic $Q_g$ is smooth away from $\ell$. Finally, by a direct  calculation the blow-up $\tilde{Q}_g$ is smooth along the exceptional divisor 
 for sufficiently general $G$.  The pencil  $|{\mathcal O}_{Q_g}(1) - 2 \ell|$  defines a rational map $(Q_g\setminus \ell) \to \Ps^1$ which can be extended to a morphism $\tilde{Q_g}\to \Ps^1$. 
The generic fiber is an irreducible conic
and, as in the proof of Lemma~\ref{lemSingTwoLines}.1,
one checks that for sufficiently general $G$
there are exactly $8$ singular fibers, each of which is a union of two lines.
Thus we can choose $G$ such that $\tilde{Q_g}$ is smooth and $$\chi(\tilde{Q_g}) =  2 \chi(\Ps^1)+8=12.$$

We put $\tilde{Q}$ (resp. $\sigma: \tilde{\Ps}^3 \rightarrow \Ps^3$) to denote the blow-up of $Q$ (resp. $\Ps^3$) along the line $\ell$. The exceptional divisor of $\sigma$ is denoted by $E$.
As argued above, the birational morphism $\tilde{Q}\to Q$ replaces the line $\ell$ 
 with a curve $E_1$ such that  $\chi(E_1)\geq 0$ (see \eqref{eq-pos-chi}). Hence
\[ \chi(Q)=\chi(\tilde{Q})+\chi(\Ps^1)-\chi(E_1) \leq 2+\chi(\tilde{Q})\]
Moreover, since $\tilde{Q}$, $\tilde{Q}_g \in |\sigma^{*}\mathcal{O}_{\Ps^3}(4)-2E|$, the surface $\tilde{Q}_g$ is smooth and $\tilde{Q}$ has only isolated singularities, we can repeat  the proof of  \cite[Corollary 5.4.4]{Dim} to obtain the equality 
$$\chi(\tilde{Q})=\chi(\tilde{Q}_g)-\mu(\tilde{Q})=12-\mu(\tilde{Q}),$$ 
where $\mu(\tilde{Q})$ denotes the sum of Milnor numbers of singularities 
of $\tilde{Q}$.

Let $s_0$ be the number of the points $P_i$ with $i \leq 4$, such that $Q$ has an (isolated) $A_1$-singularity at $P_i$, 
and let $s_1$ be the number of the isolated singular points of $\tilde{Q}$  that are  not of type $A_1$. Obviously  $\mu(\tilde{Q})\geq s_0+2s_1$. 

\noindent
By definition (see Notation~\ref{notStd}) the quartic $Q$ has $r(Q)+s(Q)$ non-nodal singularities at the triple points of the quintic threefold $X$.
Proposition~\ref{prpBasic}.\ref{p31-two} implies that  at most two of the triple points of $X$ 
are on $\ell$. This yields the inequality $s_1\geq r+s-2$ and we obtain
\[ \chi(Q)+\epsilon \leq \chi(\tilde{Q})+2+\epsilon \leq 14+\epsilon -\mu(\tilde{Q}) \leq 14+s+8r-s_0-2s_1\leq 16+7r-s_0-s_1\]
For $r=0$ we obtain \eqref{eq-maininequality}, so we can assume that $r=1$ and, by Lemma~\ref{lemm-Roneach}, at most one point $P_i$ with $i \leq 4$ belongs to $\ell$. Thus
  $s_0+s_1\geq 3$ and 
we obtain \eqref{eq-maininequality}  when $Q$ is of transversal type $A_k$ with $k \leq 2$ along $\ell$. 

Assume  that $k=3$. By Lemma~\ref{lem-ell-on-pi}  we have $\ell \not \subset \Pi_0$, so $Q$ has at least three singularities away from $\ell$. 
From Lemma~\ref{lem-atworstA3-bis}.1 we infer that the exceptional divisor $E_1$ has at most two components.  Thus Lemma~\ref{lem-atworstA3-bis}.2
implies that the generic fiber of the fibration given by the pencil  $|{\mathcal O}_Q(1) - 2\ell|$ on the variety $(Q \setminus \ell)$ has Euler number $1$ and  
the fibration  has exactly 2 singular fibers, each of  which has Euler number $1$. Thus 
$\chi(Q) = 2 + \chi(Q \setminus \ell) =  2 + 1 =3$. Since $r \leq 12$, we have 
$\chi(Q) + \epsilon  \leq 15$ and the proof of \eqref{eq-maininequality} in the case $\Sigma = \ell$ is complete.


Finally, assume that $\Sigma = \ell \cup \ell_1$,  where $\ell$, $\ell_1$ are coplanar lines.   
Suppose first that $r>0$. Then Lemma~\ref{lemm-Roneach} yields  $r=1$, i.e. $Q$ contains exactly five triple points of $X$, say $P_1$, $\ldots$, $P_5$, and the lines  $\ell$, $\ell_1$ meet in the point
$P_5$. By Proposition~\ref{prpBasic}.\ref{p31-two} the curve $\Sigma$ contains at most three triple points of $X$,  so we can assume that the points $P_3$, $P_4$ do not belong to $\Sigma$. \\
Consider the pencil $|{\mathcal O}_Q(1) - 2\ell|$. By \cite[Proposition 1.8(2)]{TopQua}, 
the quartic $Q$ is not ruled by lines, so the pencil contains a smooth conic (c.f. Remark~\ref{rem-ap2}). It has at least three singular fibers; one of them is $2\ell_1$, each of the other two contains one of the triple points $P_3$, $P_4$
(by assumption, the points $P_1$, $\ldots$, $P_4$ are coplanar, so the plane spanned by $\ell$ and $P_3$ cannot contain the point $P_4$).  By Lemma~\ref{lemSingTwoLines} there are no further singular conics in the pencil in question.

\noindent
If the pencil contains a conic that meets  $\ell$ transversally, then at most two of its members are tangent to $\ell$ (see Lemma~\ref{lemSingTwoLines}.3). Obviously $2 \ell_1$ meets the line $\ell$ in exactly one point.
 Hence $\chi(Q\setminus\ell) \leq 4$ and $\chi(Q)\leq 6$.


\noindent
 Otherwise, every member of the pencil is tangent to $\ell$,  so each singular fiber is double line (recall that each singular fiber contains a double point $\not \in \ell$) and its Euler number coincides with the one of a smooth fiber (i.e. it is $1$). Then we obtain the equalities 
$\chi(Q\setminus \ell)=2$ and $\chi(Q)=4$. Since $r=1$ and $\epsilon \leq 12$, the inequality \eqref{eq-maininequality} follows in both cases.

Assume  $r=0$. We claim that either  the pencil   $|{\mathcal O}_Q(1) - 2\ell|$  or  $|{\mathcal O}_Q(1) - 2\ell_1|$ contains a smooth conic.
Indeed, suppose that no  member of  $|{\mathcal O}_Q(1) - 2\ell|$ is a smooth conic. Then $Q$ is ruled (by lines) and,
by the description 
given in \cite[$\S$~3.2.6]{TopQua}, every line on $Q$ meets $\ell$. Moreover,  
 only finitely many lines on $Q$ meet the line $\ell_1$ (indeed, each such line runs through the point $\ell \cap \ell_1$), so the pencil 
 $|{\mathcal O}_Q(1) - 2\ell_1|$ does contain a smooth conic. \\
Finally,  as before  we can apply Lemma~\ref{lemSingTwoLines} to show that  $\chi(Q)\leq 9$. Since $r=0$ we have $\epsilon \leq 4$
and the proof of \eqref{eq-maininequality}  in the case  $\Sigma = \ell \cup \ell_1$ is complete.
\end{proof}

\noindent
{\bf Acknowledgement:}
This project was partially carried out during ``Workshop on local negativity and positivity on algebraic surfaces" at the Institute of Algebraic Geometry, Leibniz Universit\"at Hannover. We thank Matthias Sch\"utt, Roberto Laface and    Piotr Pokora for very good working conditions.  The authors would  like to thank Duco van Straten and John Ottem for discussions on quintic threefolds.

\bibliographystyle{plain}

\bibliography{quintic-26aug}

\end{document}